\documentclass{siamltex}%
\usepackage{amsmath}
\usepackage{amsfonts,amssymb}
\usepackage{amssymb}
\usepackage{graphicx}
\usepackage{amsfonts}%
\providecommand{\U}[1]{\protect\rule{.1in}{.1in}}
\newcommand{\eqdef}{\overset{\mbox{\tiny def}}{=}}

\renewcommand{\theenumi}{\roman{enumi}}
\renewcommand{\theenumiii}{\arabic{enumiii}}

\newtheorem{assumption}{Assumption}
\newtheorem{remark}[theorem]{Remark}
\newtheorem{example}[theorem]{Example}
\numberwithin{equation}{section}
\begin{document}

\title{Convergence of Numerical Time-Averaging and Stationary Measures via Poisson Equations}
\author{Jonathan C. Mattingly\thanks{Department of Mathematics, Center of Nonlinear
and Complex systems, Center for Theoretical and Mathematical Science, and
Department of Statistical Science, Duke University, Durham NC 27708-0320,
Email: jonm@math.duke.edu}
\and Andrew M. Stuart\thanks{Mathematics Institute, University of Warwick, Coventry
CV4 7AL, UK, Email: A.M.Stuart@warwick.ac.uk}
\and M.V. Tretyakov\thanks{Department of Mathematics, University of Leicester,
Leicester LE1~7RH, UK, Email: M.Tretyakov@le.ac.uk}}
\maketitle

\begin{abstract}
Numerical approximation of the long time behavior of a stochastic differential
equation (SDE) is considered. Error estimates for time-averaging estimators
are obtained and then used to show that the stationary behavior of the
numerical method converges to that of the SDE. The error analysis is based on
using an associated Poisson equation for the underlying SDE. The main
advantage of this approach is its simplicity and universality. It works
equally well for a range of explicit and implicit schemes including those with
simple simulation of random variables, and for hypoelliptic SDEs. To simplify
the exposition, we consider only the case where the state space of the SDE is
a torus and we study only smooth test functions. However we anticipate that
the approach can be applied more widely. An analogy between our approach and
Stein's method is indicated. Some practical implications of the results are discussed.

\textbf{AMS 000 subject classification. }Primary 65C30; secondary 60H35,
37H10, 60H10.


\end{abstract}

\section{Introduction}

In many application one is interested in estimating the invariant measure of
stochastic differential equation (SDE) by running a numerical scheme which
approximates the time dynamics of the SDE. Two common approaches are to use
the numerical trajectories to construct an empirical measure by the time
averaging trajectories or by averaging many different realizations to obtain a
finite ensemble average (see for example \cite{MilsteinTretyakov2007}). In
either case one produces an approximation to the numerical method's invariant
measure and is immediately presented with a number of questions including: (i)
Does the numerical scheme have a stationary measure to which it converges
quickly as time goes to infinity? (ii) How close is the numerical method's
stationary measure to the stationary measure of the underlying SDE ? (iii) How
close is the time-averaging estimator to the stationary measure of the
underlying SDE? This article mainly focuses on the second two questions.

%Replace with above -JCM
%In many applications it is needed to compute averages with respect to the
%stationary measure of a system of stochastic differential equations (SDEs) by
%running a numerical scheme which approximates the time dynamics of the SDEs.
%To compute such averages, one can exploit either time averaging, which is the
%method commonly used by the physical community and is the subject of this
%paper, or ensemble averaging (see, e.g. \cite{MilsteinTretyakov2007}). In the
%case of time averaging one is drawn to two questions: (i) Does the numerical
%scheme have a stationary measure to which it converges quickly as time goes to
%infinity? (ii) How close is the time-averaging estimator to the ergodic limit
%of the original SDEs?

The first question is addressed in many papers, including \cite{robtwe96,
Talay90,
MattinglyStuartHighamSDENUM02,LambertonPages02,KlokovVeretennikov2006}. There
are also a number of works where the second question has been considered.
%although as a consequence of addressing the first question.
%jcm - don't completely agree with this sentence or see need for it
In \cite{Talay90} it was shown that several numerical methods for SDEs,
including the forward Euler-Maruyama scheme, have unique stationary measures
which converge at the expected rate to the unique stationary measure of the
SDE. Moreover, in \cite{TalayTubaro90} an expansion of the numerical
integration error in powers of time step was obtained which allows one to use
the Richardson-Romberg extrapolation to improve the accuracy. These works
discuss the convergence in the case of relatively smooth tests functions. In
\cite{BallyTalayLESI96,BallyTalayLESII96}, the authors use techniques from
Malliavin calculus to establish smoothing properties of the discretization
scheme. This allows one to prove the convergence of the averages of test
functions which are only bounded. In turn, this shows that the distance
between the numerical method's stationary measure is close to the stationary
measure of the SDE in the total-variation norm.
%AMS reorder
In \cite{ShardlowStuartPENA00}, a general method of deducing closeness of the
stationary measure from an error estimate on a finite time interval is
proposed, though in general it does not provide optimal estimates of the
convergence rate. In the earlier works a global Lipschitz assumption was
imposed on the coefficients of the SDE which was lifted in
\cite{RobertsTweedie96,MattinglyStuartHighamSDENUM02,TalayHam02,BouVan09}. The
papers \cite{Talay90,TalayTubaro90} deal with elliptic SDEs, while the
hypoelliptic case is treated in
\cite{BallyTalayLESI96,BallyTalayLESII96,TalayHam02,MattinglyStuartHighamSDENUM02,Lem07,BouVan09}%
.

%modified the next section a little - jcm
In this paper we obtain error estimates for time-averaging estimators. The
error estimates are given in terms of a term of the same order of magnitude as
the weak finite time error of the numerical integrator used and the length of
the simulation. Control of the time-averages is then used to prove closeness
of the numerical method's stationary measure to the stationary measure of the
SDE. The convergence rate obtained is the same as the weak convergence rate
obtained on a finite time interval. This highlights the general principle
that, for time-dependent problems, finite time approximation properties can be
transferred to infinite time approximation properties in the presence of
suitable stability. Here the underlying stability is the geometric ergodicity
of the SDE.
%AMS changed words

%This highlights the fact that we are
%dealing with numerical consistence and stability on an infinite time interval.
%As discussed below, we have picked the setting of a compact phase space so
%stability is immediate. This note shows that under reasonable assumptions
%finite time consistency implies infinite time consistency.
%jcm -added last sentences

The main tool in our analysis is an associated Poisson equation for the
underlying SDE. In this error analysis, the Markov chain generated by the
numerical method need not be uniquely ergodic. It is shown that any stationary
measure of the numerical method will be close to the unique stationary measure
of the underlying SDE. We see the main advantage of the current analysis being
its simplicity and universality. In particular, we obtain the error estimates
in the case of numerical schemes with any reasonable simple random variables
(including the discrete random variables used in weak Euler schemes
\cite{MilsteinTretyakov2004}). Furthermore, our approach works equally well
for a range of explicit and implicit schemes, and, perhaps more importantly,
it works for hypoelliptic SDEs. It is comparatively short and the analysis
leverages classical results on PDEs. It should be possible to carry over this
approach to methods with adaptive step-sizes (see
\cite{LambaMattinglyStuart07,Lem07}). It should also be adaptable to the SPDE
setting, using the Poisson equation in infinite dimensions (see related
applications of Poisson equations in \cite{CerraiFreidlin09,DaPratoZabczyk02}).

Our goal here has not been to give the most general results. We have picked
the simplifying setting of a compact phase space, namely the $d$-dimensional
torus, and relatively smooth test functions. The extension to the whole space
requires control of the time spent outside of the center of the phase space.
Under additional assumptions in the spirit of
\cite{MattinglyStuartHighamSDENUM02}, we believe that versions of the present
results can be proven in $\mathbf{R}^{d}$.

In Section~\ref{sec:sde}, we introduce the basic setting for the underlying
SDE and discuss the hypoellipticity assumptions. In
Section~\ref{sec:numerical}, we describe the class of numerical approximations
we consider. We give examples of explicit and implicit methods which fit our
framework. In Section~\ref{sec:poisson}, we discuss the auxiliary Poisson
equation which will be used to prove the main results, and we give properties
of its solution. In Section~\ref{sec:LLN}, we warm up by showing how to prove
a law of large numbers, for the SDE,
%AMS extra words
using an auxiliary Poisson equation. Section~\ref{sec:Closeness} contains the
main results of the article which give a number of senses in which the
numerical time-averaging estimators are close to the corresponding stationary
time average of the SDE. In Section~\ref{sec:high}\ we extend the results of
Section~\ref{sec:Closeness} to numerical methods of higher orders and in
Section~\ref{sec:expan} the Richardson-Romberg (Talay-Tubaro) error expansion
is considered. Section~\ref{sec:ClosenessII} uses the results of
Section~\ref{sec:error} to give a quantitative estimate on the distance of the
numerical stationary measures from the underlying stationary measure for the
SDE. In Section~\ref{sec:stein} we make some general comments about analogies
with Stein's method for proving the convergence of distributions to a limiting
law. Some practical implications of the results of Section~\ref{sec:error} are
discussed in Section~\ref{sec:practice}, where we classify the errors of the
numerical time-averaging estimators and, in particular, pay attention to the
statistical error.

\section{SDE Setting}

\label{sec:sde} Let $(\Omega,\mathcal{F},\mathcal{F}_{t},\mathbf{P}),$
$t\geq0,$ be a filtered probability space and $W(t)=(W_{1}(t),\ldots
,W_{m}(t))^{\top}$ be an $m$-dimensional $\{\mathcal{F}_{t}\}_{t\geq0}%
$-adapted standard Wiener process. Consider the Markov process $X(t)$ on the
torus\footnote{In physical applications a process on the torus can be of
interest when periodic boundary conditions imposed on the problem. A typical
example is a noisy gradient system which may be used to sample from the Gibbs'
distribution
%AMS reword and add ref
with a periodic potential \cite{StuartPavliotis08}. We also note that our
results can potentially be applied to approximating Lyapunov exponents on
compact smooth manifolds \cite{GrorudTalay96}.} $\mathbf{T}^{d}$ whose time
evolution is governed by the Ito SDE
\begin{equation}
dX(t)=f(X(t))\,dt+g(X(t))\,dW(t)\ ,\label{eq:sde}%
\end{equation}
where $f=(f_{1},\ldots,f_{d})^{\top}\colon$ $\mathbf{T}^{d}\rightarrow
\mathbf{R}^{d}$ and $g\colon$ $\mathbf{T}^{d}\rightarrow\mathbf{R}^{d\times
m}$. We assume that $f$ and $g$ are Lipschitz continuous (and hence uniformly
bounded) functions on $\mathbf{T}^{d}$. Under these assumptions, equation
\eqref{eq:sde} has a unique pathwise global solution. The generator of the
Markov process $X(t)$ is
%AMS inserted index definition of generator%
\begin{align}
\mathcal{L} &  \eqdef f.\nabla+\frac{1}{2}a\colon\!\!\nabla\nabla\ \label{eq:genSDE}%
\\
&  =\sum_{i=1}^{d}f_{i}\frac{\partial}{\partial x_{i}}+\frac{1}{2}\sum
_{i,j=1}^{d}a_{ij}\frac{\partial^{2}}{\partial x_{i}\partial x_{j}%
},\,\nonumber
\end{align}
where $a(x)=g(x)g^{\top}(x)$. For a twice differentiable function $\phi
\colon\mathbf{T}^{d}\rightarrow\mathbf{R}$ and $x\in\mathbf{T}^{d}$ we have
\begin{align*}
(\mathcal{L}\phi)(x) &  =f(x).\nabla\phi(x)+\frac{1}{2}a(x):\nabla\nabla
\phi(x)\\
&  =\sum_{i=1}^{d}f_{i}(x)\frac{\partial\phi}{\partial x_{i}}(x)+\frac{1}%
{2}\sum_{i,j=1}^{d}a_{ij}(x)\frac{\partial^{2}\phi}{\partial x_{i}\partial
x_{j}}(x)\,.
\end{align*}
The operator $\mathcal{L}$ generates a strong Markov semigroup $P_{t}$, for
$t\geq0$, which maps smooth bounded functions to smooth bounded functions. It
can also be defined by $(P_{t}\phi)(x)=\mathbf{E}_{x}\phi(X_{t})$ where
$X_{0}=x$. By duality, $P_{t}$ also induces a semigroup which acts on
probability measures. In the name of notational economy, we will also denote
this semigroup by $P_{t}$.

The $k$-th column of $g$, which we denote by $g^{(k)}$, can be viewed as a
function from $\mathbf{T}^{d}\rightarrow\mathbf{R}^{d}$. Hence $\{f,g^{(1)}%
,\dots,g^{(m)}\}$ is a collection of vector fields on $\mathbf{T}^{d}$ and
\eqref{eq:sde} can be rewritten as
\[
dX(t)=f(X(t))\,dt+\sum_{k=1}^{m}g^{(k)}(X(t))\,dW_{k}(t).
\]
%AMS reword
This form of equation \eqref{eq:sde} makes it clear that there is a
deterministic drift in the direction of $f$ and independent random kicks in
each of $m$ directions $\{g^{(1)},\dots,g^{(m)}\}$. We will require that the
randomness injected into the dynamics in the $g^{(k)}$ directions effects all
directions sufficiently to produce a \textquotedblleft
smoothing\textquotedblright\ effect at the level of probability densities.
Uniform ellipticity is the simplest assumption which ensures \textquotedblleft
sufficient spreading\textquotedblright\ of the randomness. However, it is far
from necessary and many important examples are not uniformly elliptic. For our
purposes, it is sufficient that the system is hypoelliptic.

The conditions given in our standing Assumption~\ref{ass:standing}, given
below, are enough to ensure hypoellipticity. To state them, we need to recall
the definition of the Lie-bracket (or commutator) of two vector fields. In our
setting, given two vector fields $h,\tilde{h}$ on $\mathbf{T}^{d}$ with
$h=(h_{1},\dots,h_{d}) ^{\top}$ and $\tilde{h}=(\tilde{h}_{1},\dots,\tilde
{h}_{d}) ^{\top}$ one defines the Lie-bracket as $[h,\tilde{h}%
](x)\overset{\mbox{\tiny def}}{=}(h.\nabla\tilde{h})(x)-(\tilde{h}.\nabla
h)(x)$. Hence the $j$th component of $[h,\tilde{h}](x)$ is given by
$\sum_{i=1}^{d}\big(h_{i}(x)\frac{\partial\tilde{h}_{j}}{\partial x_{i}%
}(x)-\tilde{h}_{i}(x)\frac{\partial h_{j}}{\partial x_{i}}(x)\big)$. The
vector $[h,\tilde{h}](x)$ may be thought of as the new direction generated by
infinitesimally following $h$, then $\tilde{h}$, then $-h$ and finally
$-\tilde{h}$. Our most general assumption will be that the collection of all
the brackets generated by the randomness spans the tangent space at all
points. To track how the noise spreads, we introduce the following increasing
set of vector fields
\[
\Lambda_{0}=\mathrm{span}\{f,g^{(1)},\ldots,g^{(m)}\},\ \ \Lambda
_{n+1}=\mathrm{span}\big\{h,\,[\bar{h},h]:\ h\in\Lambda_{n},\ \bar{h}%
\in\Lambda_{0}\big\}\ .
\]

The following is our basic assumption governing the stationary measure of the
system, and its smoothness.

\vspace{1ex}

\begin{assumption}
We assume that one of the following two assumptions hold:\label{ass:standing}

\begin{enumerate}
\item \label{ass1i}(Elliptic Setting) The matrix-valued function
$a(x)=g(x)g^{\top}(x)$ is uniformly positive definite: there exists $\alpha>0$
so that, for all $z\in\mathbf{R}^{d}$ and $x\in\mathbf{T}^{d}$, $\alpha
|z|^{2}\leq a(x)z.z\ .$
%\item  $f$ and $g$ are $C^\infty(\T^d,\R^d)$ and such that
%$\Span(\Lambda_n)=\R^d$ for some $n$.

\item \label{ass1ii} (Hypoelliptic setting) The functions $f$ and $g$ are
$C^{\infty}(\mathbf{T}^{d},\mathbf{R}^{d})$ and such that, for some $n$ and
all $x\in\mathbf{T}^{d}$, $\Lambda_{n}=\Lambda_{n}(x)=\mathbf{R}^{d}$ and
\eqref{eq:sde} possesses a unique stationary measure.
\end{enumerate}
\end{assumption}

\vspace{1ex}

In either case in Assumption~\ref{ass:standing}, we have that \eqref{eq:sde}
has a unique stationary measure, henceforth denoted by $\mu$, which has a
density with respect to Lebesgue measure on ${\mathbf{T}}^{d}$.
%AMS mentioning support of the stationary  measure

As already mentioned, the goal of this paper is to give a simple, yet robust,
proof that time averages obtained using a class of numerical methods for
simulating \eqref{eq:sde} is close to the corresponding ergodic limits of
\eqref{eq:sde}. To this end, we now describe the class of numerical methods we
will consider.

\section{Numerical approximations}

\label{sec:numerical} We will begin by considering a class of simple numerical
approximation of \eqref{eq:sde} given by the generalized Euler-Maruyama method
on the torus $\mathbf{T}^{d}:$
\begin{equation}
\left\{
\begin{aligned} X_{n+1}&= X_n + F(X_n,\Delta )\,\Delta + G(X_n,\Delta )\eta_{n+1}\sqrt{\Delta } \, , \\ X_0 &= x_0 \, , \end{aligned}\right.
\label{eq:euler}%
\end{equation}
where $\Delta$ is the time increment, $F\colon\mathbf{T}^{d}\times
(0,1)\rightarrow\mathbf{R}^{d}$, $G\colon\mathbf{T}^{d}\times(0,1)\rightarrow
\mathbf{R}^{d\times m}$, and $\eta_{n}=(\eta_{n,1},\ldots,\eta_{n,m})^{\top}$
is an $\mathbf{R}^{m}$-valued random variable and $\{\eta_{n,i}:n\in
\mathbf{N},i\in\{1,\ldots,m\}\}$ is a collection of i.i.d. real-valued random
variables satisfying
\[
\mathbf{E}\eta_{n,i}=\mathbf{E}\eta_{n,i}^{3}=0,\qquad\mathbf{E}\eta_{n,i}%
^{2}=1,\qquad\mathbf{E}\eta_{n,i}^{2r}<\infty\,
\]
for a sufficiently large $r\geq2.\footnote{We do not specify $r$ in the
statements of the forthcoming theorems since common numerical schemes use
random variables with bounded moments up to any order. At the same time, in
each proof of Section \ref{sec:Closeness} it is not difficult to recover the
number of bounded moments required.}$ More general methods will be
considered in Section~\ref{sec:high}.

For example, $\eta_{1,1}\sim\mathcal{N}(0,1)$ satisfies these assumptions as
well as $\eta_{1,1}$ with the law
\begin{equation}
\mathbf{P}(\eta_{1,1}=\pm1)=1/2\ . \label{eq:dis}%
\end{equation}
We also note that the above two choices of $\eta_{1,1}$ guarantee existence of
finite moments of $\eta_{1,1}$ of any order. Further, we assume that $\eta
_{n},$ $n\in\mathbf{N},$ is defined on the probability space $(\Omega
,\mathcal{F},\mathbf{P})$ and is $\{\mathcal{F}_{t_{n}}\}$-adapted, for
$t_{n}=n\Delta.$

Since we want the dynamics of \eqref{eq:euler} to be \textquotedblleft
close\textquotedblright\ to those of \eqref{eq:sde}, we make the following
assumption on the \textquotedblleft local error\textquotedblright\ of \eqref{eq:euler}.

\vspace{1ex}

\begin{assumption}
\label{ass:eular} For some $C >0$, all $x \in\mathbf{T}^{d}$, and all $\Delta$
sufficiently small
\begin{equation}
\label{goodMethod}| F(x,\Delta) -f(x)| + | G(x,\Delta)-g(x)| \leq C \Delta\,.
\end{equation}

\end{assumption}

\vspace{1ex}

Under this assumptions we expect that $X_{n} \approx X(t_{n})$ where as before
$t_{n}=n\Delta$. In fact it follows from the general theory
\cite{MilsteinTretyakov2004} that such numerical schemes have first-order weak
convergence on finite time intervals.

In analogy to \eqref{eq:genSDE}, we define an operator associated to
\eqref{eq:euler} by
\[
\mathcal{L}^{\Delta}(\,\cdot\,)=F(\,\cdot\,,\Delta).\nabla+\frac{1}%
{2}A(\,\cdot\,,\Delta)\colon\!\!\nabla\nabla\ ,
\]
where $A(x,\Delta)=G(x,\Delta)G^{\top}(x,\Delta)$. While not the generator of
the Markov process \eqref{eq:euler}, it is the generator's leading
order part since
\begin{equation}
\mathbf{E}\Big(\phi(X_{1})-\phi(x_{0})\Big)=(\mathcal{L}^{\Delta}\phi
)(x_{0})\Delta+O(\Delta^{2})\quad\text{ as }\quad\Delta\rightarrow0\,.
\label{eq:SDEexpand}%
\end{equation}
Since (by Assumption~\ref{ass:eular}) we have that
\begin{equation}
|\mathcal{L}^{\Delta}\phi-\mathcal{L}\phi|_{\infty}\leq C\Delta|D^{2}%
\phi|_{\infty}\ , \label{eq:genapp}%
\end{equation}
where $D^{2}\phi$ is the second derivative, we deduce that
\begin{equation}
\mathbf{E}\Big(\phi(X_{1})-\phi(x_{0})\Big)=(\mathcal{L}\phi)(x_{0}%
)\Delta+O(\Delta^{2})\quad\text{ as }\quad\Delta\rightarrow0\,.
\label{eq:genApprox}%
\end{equation}
It is reasonable to expect that the distribution of the dynamics of the SDE
and its approximation will be close to each other since the leading part of
the generator of the approximation \eqref{eq:euler} is close to that of the
original process \eqref{eq:sde}.

We close this section by giving two approximation methods which satisfy
Assumption~\ref{ass:eular}.

\begin{example}
[Explicit Euler-Maruyama]We define
\begin{equation}
X_{n+1}=X_{n}+f(X_{n})\Delta+g(X_{n})\eta_{n+1}\sqrt{\Delta} \label{Euler-Mar}%
\end{equation}
and hence $F(x,\Delta)=f(x)$ and $G(x,\Delta)=g(x)$.
\end{example}

\begin{example}
[Implicit split-step]We define
\begin{align*}
X_{n+1}^{\ast}  &  =X_{n}+f(X_{n+1}^{\ast})\Delta\\
X_{n+1}  &  =X_{n+1}^{\ast}+g(X_{n+1}^{\ast})\eta_{n+1}\sqrt{\Delta}%
\end{align*}
and hence $F(x,\Delta)=f(y)$ and $G(x,\Delta)=g(y)$ where $y$ is and element
of $\{y\in\mathbf{T}^{d}:y=x+f(y)\Delta\}$ which is closest to $x$.
\end{example}

\begin{remark}
\label{rmk:localError} From \eqref{eq:genApprox} and the backwards Kolmogorov
equation for \eqref{eq:sde}, Assumption~\ref{ass:eular} is equivalent to
\begin{equation}
\mathbf{E}\phi(X_{1})-\mathbf{E}\phi(X(\Delta))=O(\Delta^{2}) \label{locer}%
\end{equation}
for all sufficiently smooth $\phi$, provided that $X_{0}=X(0)$. In the
language of consistency and stability of a numerical method used in
the introduction, \eqref{eq:genApprox} and \eqref{locer} provide an appealing way to
characterize the degree of local consistency of a numerical method. This is will be
the starting point for our discussion of higher order methods in Section~\ref{sec:high}.
\end{remark}

\section{Poisson equation}

\subsection{Background}

\label{sec:poisson} It is a \textquotedblleft meta-theorem\textquotedblright%
\ in averaging, homogenization and ergodic theory that if one can solve the
relevant Poisson equation then one can prove results about the desired limit
of a time-average. A central theme of this paper is to show how an appropriate
Poisson equation can be used to analyze the long time average of a given
function $\phi\colon\mathbf{T}^{d}\rightarrow\mathbf{R}$, evaluated along a
numerical approximation of \eqref{eq:sde} and obtain information about its
closeness to the corresponding ergodic limits.

In this section, for motivation, we illustrate key ideas related to the
Poisson equation, purely in the continuous time setting. To this end,
recalling that $\mu$ is the unique stationary measure of \eqref{eq:sde} and
$\mathcal{L}$ its generator, given $\phi\colon\mathbf{T}^{d}\rightarrow
\mathbf{R}$ we define the stationary average of $\phi$ by
%AMS added domain of integration and changed domain of phi above%
\begin{equation}
\bar{\phi}=\int_{\mathbf{T}^{d}}\phi(z)\,\mu(dz) \label{eq:ergli}%
\end{equation}
and let $\psi$ solve the Poisson equation
\begin{equation}
\mathcal{L}\psi=\phi-\bar{\phi}\,. \label{eq:poss}%
\end{equation}
Under Assumption~\ref{ass:standing}, \eqref{eq:poss} possesses a unique
solution which is at least as smooth as $\phi$. For $k\in\mathbf{N}$, we
denote by $W^{k,\infty}$ the space of functions from $\mathbf{T}^{d}$ to
$\mathbf{R}$ such that the function and all of its partial derivatives up to
order $k$ are essentially bounded. We then have the following result whose
proof we sketch.
%jcm - remove. not exactly the theorem given here.
%details can be found in the papers cited.

\begin{theorem}
\label{thm:poison} Under Assumption~\ref{ass:standing}-\ref{ass1i}), given any
$\phi\in W^{k,\infty}$, with $k\in\mathbf{N}\cup\{0\}$ there exists a unique
solution $\psi\in W^{k+2,\infty}$ to \eqref{eq:poss}. Under
Assumption~\ref{ass:standing}-\ref{ass1ii}) there exists a $\delta>0$ so that
given any $\phi\in W^{k,\infty}$, with $k\in\mathbf{N}\cup\{n \geq2\}$, there
exists a unique solution $\psi\in W^{k+\delta,\infty}$ to \eqref{eq:poss}.
\end{theorem}

\begin{proof}
Let $\triangle$ denote the Laplacian on the $d$ dimensional torus. In either
the hypoelliptic or elliptic case, one knows that for some positive $\alpha$
and $C$
\[
\Vert(-\triangle)^{\alpha}\phi\Vert_{L^{2}}\leq C(\Vert\mathcal{L}\phi
\Vert_{L^{2}}+\Vert\phi\Vert_{L^{2}})
\]
for all smooth $\phi$. (See \cite{HormanderIV94} and recall that our space is
compact.) This implies that $\mathcal{L}$ has compact resolvent and discrete
spectrum (consisting of isolated points). Thus $\mathcal{L}$ has a spectral
gap and is invertible on the complement of the span of the eigenfunctions with
zero eigenvalue. The space is nothing other than the functions which have zero
mean with respect to an invariant measure for the corresponding semigroup
$P_{t}$ (i.e. $\mu$ such that $P_{t}\mu=\mu$ or in terms of the generator
$\mathcal{L}^{\ast}\mu=0$). Since the system has a unique invariant measure
due to Assumption~\ref{ass:standing}, we see that $\phi-\bar{\phi}$ has zero
projection onto the space spanned by eigenfunctions with eigenvalue zero.
Hence we know there exists a function $u$ which solves \eqref{eq:poss} weakly,
by the Fredholm alternative.

This leaves the regularity. In the uniformly elliptic case, see
\cite{Hasminskii1980,Krylov96Hol,StroockPDE08}. In the hypoelliptic case the
results follow from Theorem 7.3.4 of \cite{StroockPDE08} or
\cite{HormanderIV94}. This states that there is $\delta>0$ such that, if
$\phi\in W^{k,p},$ then $u\in W^{k+\delta,p}$ for $p\in\lbrack2,\infty)$ and
$k\in\mathbf{N}$. Since our $\phi\in W^{k,p}$ for all $p\geq2$, we know that
$u\in\bigcap_{p\geq1}W^{k+\delta,p}$. Since our space has finite measure, we
know that $W^{k+\delta,\infty}=\bigcap_{p\geq2}W^{k+\delta,p}$.

%In both cases the Markov semigroup $P_{s}=\exp\bigl(\mathcal{L}s\bigr)$ has a
%spectral gap. Hence
%\begin{equation}
%\label{eq:solPoisson}\psi(x)\overset{\mbox{\tiny def}}{=}-\int_{0}^{\infty
%}P_{s}(\phi-\bar{\phi})(x)ds
%\end{equation}
%is well defined and is a distribution-valued
%solution to \eqref{eq:poss}. This
%leaves the regularity. In the uniformly elliptic case see
%\cite{Hasminskii1980,Krylov96Hol,StroockPDE08}. In the hypoelliptic case the
%results follow from Theorem 7.3.4 of \cite{StroockPDE08} or
%\cite{HormanderIV94}. This states that there is $\delta>0$ such that, if
%$\phi\in W^{k,p},$ then $u\in W^{k+\delta,p}$ for $p\in\lbrack2,\infty)$ and
%$k\in\mathbf{N}$. Since our $\phi\in W^{k,p}$ for all $p\geq2$, we know that
%$u\in\bigcap_{p\geq1}W^{k+\delta,p}$. Since our space has finite measure, we
%know that $W^{k+\delta,\infty}=\bigcap_{p\geq2}W^{k+\delta,p} $.

\end{proof}

\begin{remark}
One of the principle technical issues that must be addressed in order to
extend the results in this paper from the $\mathbf{T}^{d}$ case to general
ergodic SDE on $\mathbf{R}^{d}$ is proof of the existence of nice,
well-controlled solutions to the above Poisson equation. Many of the needed
results can be found in
\cite{pardoux_poisson_2001,pardoux_poisson_2003,pardoux_poisson_2005}.
\end{remark}

\subsection{A strong law of large numbers: an illustrative example}

\label{sec:LLN} We now show how to use the Poisson equation to prove a law of
large numbers for \eqref{eq:sde}. This idea appears frequently in the
literature \cite{bhattacharya_1982,
EthierKurtz86,Papanicolaou99,Pag01,StuartPavliotis08}. Nonetheless since this
technique is central to our investigation of discrete time approximations to
diffusions, we first highlight the main ideas in continuous time. However,
before we begin it is worth noting that, at least formally, the solution to
Poisson equation can be written as
\begin{equation}
\psi(x)\overset{\mbox{\tiny def}}{=}-\int_{0}^{\infty}P_{s}(\phi-\bar{\phi
})(x)ds\,. \label{eq:solPoisson}%
\end{equation}
This can be interpreted as the total fluctuation over time of $P_{t}\phi$ from
$\bar{\phi}$; and hence, it is not surprising that $\psi$ can be used to
control convergence to equilibrium.
%jcm added last sentence.

As stated in Theorem~\ref{thm:poison} when Assumption~\ref{ass:standing}
holds, given any $\phi\colon$ $\mathbf{T}^{d}\rightarrow\mathbf{R}$ with
$\phi\in W^{2,\infty}$ there exists a unique $\psi\colon$ $\mathbf{T}%
^{d}\rightarrow\mathbf{R}$ which solves \eqref{eq:poss}. Furthermore $\psi\in
W^{2,\infty}$ and hence It\^{o}'s formula then tells us that
\begin{equation}
\psi(X(t))-\psi(x_{0})=\int_{0}^{t}\big(\phi(X(s))-\bar{\phi}\big)\,ds+\int%
_{0}^{t}(\nabla\psi)(X(s))\,.\,g(X(s))dW(s)\ . \label{eq:ito}%
\end{equation}
Rearranging this, we obtain that
\begin{equation}
\frac{1}{t}\int_{0}^{t}\phi(X(s))\,ds-\bar{\phi}=\frac{\psi(X(t))-\psi(x_{0}%
)}{t}-\frac{1}{t}M(t)\ , \label{eq:LLNPeqn}%
\end{equation}
where $M(t)=\int_{0}^{t}(\nabla\psi)(X(s))\,.\, g(X(s))\,dW(s)$. Since $\psi$
is bounded, the first term on the right-hand side goes to zero as
$t\rightarrow\infty$. To see that the last term also goes to zero as
$t\rightarrow\infty$ observe that
\[
\frac{1}{t^{2}}\mathbf{E}\big(M(t)^{2}\big)=\frac{1}{t^{2}}\mathbf{E}\langle
M\rangle(t)\leq\frac{K}{t}%
\]
for some $K>0$ independent of time since $\ \nabla\psi$ and $g$ are both
bounded on $\mathbf{T}^{d}$. More precisely, we have shown that for any
initial $x_{0}$
\begin{equation}
\mathbf{E}\Big(\frac{1}{T}\int_{0}^{T}\phi(X(s))ds-\bar{\phi}\Big)^{2}%
\leq\frac{K}{T}\,, \label{exampleE}%
\end{equation}
which is a quantitative version of what is often called the mean ergodic theorem.

We can view $\frac{1}{T}\int_{0}^{T}\phi(X(s))\,ds$ as an estimator for
$\bar{\phi}$ and it follows from (\ref{eq:LLNPeqn}) that
\begin{equation}
\mathrm{Bias}\left(  \frac{1}{T}\int_{0}^{T}\phi(X(s))\,ds\right)
\overset{\mbox{\tiny def}}{=}\mathbf{E}\Big(\frac{1}{T}\int_{0}^{T}%
\phi(X(s))ds-\bar{\phi}\Big)=O\left(  \frac{1}{T}\right)  . \label{eq:bias1}%
\end{equation}
One can also show that
\begin{equation}
\mathrm{Var}\left(  \frac{1}{T}\int_{0}^{T}\phi(X(s))\,ds\right)  =O\left(
\frac{1}{T}\right)  \ . \label{eq:var1}%
\end{equation}
The estimate (\ref{eq:var1}) easily follows from (\ref{exampleE}) and
(\ref{eq:bias1}) but can also be obtained directly from (\ref{eq:LLNPeqn})
using an additional mixing condition.

To obtain an almost sure (a.s.) result\footnote{This result can also be
obtained using Kronecker's lemma (see, e.g. \cite{Elliot01}). Though
the approach we follow gives explicit error estimates which can be useful. }, for any
$\varepsilon>0$ we introduce $A({T;}\varepsilon)=\{\frac{1}{T}\sup_{t\leq
T}|M(t)|>T^{\varepsilon-\frac{1}{2}}\}.$ By the Doob inequality for continuous
martingales and the fact that $\mathbf{E}|M(T)|^{2}\leq KT$, we get
\[
\mathbf{P}(A({T;}\varepsilon))\leq\frac{E|M(T)|^{2}}{T^{1+2\varepsilon}}%
\leq\frac{K}{T^{2\varepsilon}}\ .
\]
Hence
\[
\sum_{n\in\mathbf{N}}\mathbf{P}(A({2^{n};}\varepsilon))<\infty\;
\]
and the Borel-Cantelli lemma implies that there exists an a.s. bounded random
variable $C(\omega)>0$ and $n_{0}$ so that for every $n\in\mathbf{N}$ with
$n\geq n_{0}$ one has
\[
\frac{1}{2^{n}}\sup_{t\leq2^{n}}|M(t)|\leq\frac{C(\omega)}%
{2^{n(1/2-\varepsilon)}}\ .
\]
Hence with probability one, for every $t\in\lbrack2^{n},2^{n+1})$ with $n\geq
n_{0}$ one has
\[
\frac{1}{t}|M(t)|\leq\frac{1}{2^{n}}\sup_{2^{n}\leq s\leq2^{n+1}}%
|M(s)|\leq\frac{2}{2^{n+1}}\sup_{s\leq2^{n+1}}|M(s)|\leq\frac{2C(\omega
)}{2^{(n+1)(1/2-\varepsilon)}}\leq\frac{2C(\omega)}{t^{1/2-\varepsilon}}\ .
\]
By combining this estimate with \eqref{eq:LLNPeqn} and the fact that $\psi$ is
bounded, we see that for any $\varepsilon>0$ and for every $t\geq2^{n_{0}}$
one has
\begin{equation}
\Big|\frac{1}{T}\int_{0}^{T}\phi(X(s))\,ds-\bar{\phi}\Big|\leq\frac
{2|\psi|_{\infty}}{T}+\frac{C(\omega)}{T^{1/2-\varepsilon}}\qquad
\text{a.s.}\label{eq:ascon}%
\end{equation}
for some a.s. bounded $C(\omega)>0.$ We note that (\ref{eq:ascon}) implies
that for any initial $x_{0}$
\begin{equation}
\lim_{t\rightarrow\infty}\frac{1}{t}\int_{0}^{t}\phi(X(s))\,ds=\bar{\phi
}\qquad\text{a.s.}\;.\label{exampleAS}%
\end{equation}
In other words, the strong law of large number holds starting from any initial
$x_{0}$. Our assumptions are sufficient to ensure that \eqref{eq:sde} has a
unique stationary measure $\mu$. Hence, it follows from Birkoff's ergodic
theorem that \eqref{exampleAS} holds for $\mu$-a.e. initial $x_{0}$. The above
argument not only shows that the result holds for \emph{all} $x_{0},$ but also
gives quantitative estimates on the rate of convergence.
%jcm

\section{Error analysis for the numerical time-average\label{sec:error}}

\subsection{First order accurate schemes}

\label{sec:Closeness} We now wish to generalize the calculations in the
previous section to prove the closeness of the time averages obtained via the Euler
approximation \eqref{eq:euler} to the stationary averages of \eqref{eq:sde}.
We consider sample path estimates in this section and then convert these to
distance estimates in an appropriate metric, in Section~\ref{sec:ClosenessII}.

Let $T=N\Delta.$ Introduce the estimator (discrete time-average) $\hat{\phi
}_{N}$ for the stationary average $\bar{\phi}$:
%jcm - ergodic limit
%seem less clear%
\begin{equation}
\hat{\phi}_{N}\overset{\mbox{\tiny def}}{=}\frac{1}{N}\sum_{n=0}^{N-1}%
\phi(X_{n})\ , \label{eq:estim}%
\end{equation}
where $X_{n}$ is the SDE approximation defined in \eqref{eq:euler}.

We start by proving a mean convergence result, followed by $L^{2}$ and almost
sure theorems.

\begin{theorem}
\label{thm:veryBasicConv} Let Assumptions~\ref{ass:standing} and
\ref{ass:eular} hold. Let $H$ denote $W^{2,\infty}$ in the elliptic setting
and $W^{4,\infty}$ in the hypoelliptic setting. Then for any $\phi\in H$, one
has
\begin{equation}
\Big|\mathbf{E}\hat{\phi}_{N}-\bar{\phi}\Big|\leq C\Big( \Delta+\frac{1}%
{T}\Big),
\end{equation}
where $T=N\Delta$ and $C$ is some positive constant independent of $\Delta$
and $T$. Furthermore, the constant $C$ is a linear function of $|\phi|_{H}$
and otherwise independent of $\phi$.\footnote{We indicate the dependence of
constants on $\phi$ since this dependence is used in
Section~\ref{sec:ClosenessII}. Obviously, all the constants appearing in the
statements of this and forthcoming theorems also depend on the coefficients of
the SDE (\ref{eq:sde}) and on the numerical method used.}
%AMS reword footnote

\end{theorem}

\begin{proof}
Under either of the conditions from Assumption~\ref{ass:standing}, we know
from Theorem~\ref{thm:poison} that the solution $\psi$ of \eqref{eq:poss} is
in $W^{4,\infty}$. In the interest of clarity and definiteness, we will
proceed under the first condition in Assumption~\ref{ass:standing} (the
elliptic setting) where the $|D^{4}\psi|_{\infty}\leq C |D^{2}\phi|_{\infty}$,
$|D^{3}\psi|_{\infty}\leq C |D\phi|_{\infty}$, $|D^{2}\psi|_{\infty}\vee
|D\psi|_{\infty}\vee|\psi|_{\infty}\leq C |\phi|_{\infty}$. In the
hypoelliptic setting, the $k^{th}$ derivative of $\psi$ is bounded by the same
derivative of $\phi$, necessitating the change in definition of $H$ between
the elliptic and hypoelliptic cases in the statement of the theorem.
%AMS rerworded above

For brevity, we write $\phi_{n}=\phi(X_{n})$, $F_{n}=F(X_{n},\Delta)$,
$G_{n}=G(X_{n},\Delta)$, $\psi_{n}=\psi(X_{n})$ and $D^{k}\psi_{n}=(D^{k}%
\psi)(X_{n})$ where $(D^{k}\psi)(z)$ is the $k$th derivative. We write
$(D^{k}\psi)(z)[h_{1},\ldots,h_{k}]$ for the derivative evaluated in the
directions $h_{j}$. Defining $\bar{\delta}_{n}\overset{\mbox{\tiny def}}{=}
X_{n+1}-X_{n}=\Delta F_{n}+\sqrt{\Delta}G_{n}\eta_{n+1}$, we have
\begin{equation}
\psi_{n+1}=\psi_{n}+D\psi_{n}[\bar{\delta}_{n}]+\frac{1}{2}D^{2}\psi_{n}%
[\bar{\delta}_{n},\bar{\delta}_{n}]+\frac{1}{6}D^{3}\psi_{n}[\bar{\delta}%
_{n},\bar{\delta}_{n},\bar{\delta}_{n}]+R_{n+1}, \label{taylor}%
\end{equation}
where
\[
R_{n+1}=\Big(\frac{1}{6}\int_{0}^{1}s^{3}D^{4}\psi(sX_{n}+(1-s)X_{n+1}%
)\,ds\Big)[\bar{\delta}_{n},\bar{\delta}_{n},\bar{\delta}_{n},\bar{\delta}%
_{n}]
\]
is the remainder given by the Taylor theorem. Hence,
\begin{align}
\psi_{n+1}  &  =\psi_{n}+\Delta\mathcal{L}^{\Delta}\psi_{n}+\Delta^{\frac
{1}{2}}D\psi_{n}[G_{n}\eta_{n+1}]+\Delta^{\frac{3}{2}}D^{2}\psi_{n}%
[F_{n},G_{n}\eta_{n+1}]\label{taylorAgain}\\
&  +\frac{1}{2}\Delta^{2}D^{2}\psi_{n}[F_{n},F_{n}]+\frac{1}{6}D^{3}\psi
_{n}[\bar{\delta}_{n},\bar{\delta}_{n},\bar{\delta}_{n}]+r_{n+1}%
+R_{n+1},\nonumber
\end{align}
where
\[
r_{n+1}=\frac{\Delta}{2}\Big(D^{2}\psi_{n}[G_{n}\eta_{n+1},G_{n}\eta
_{n+1}]-A(x,\Delta)\colon\!\!\nabla\nabla\psi_{n}\Big)\,.
\]
Summing \eqref{taylorAgain} over the first $N$ terms, dividing by $N\Delta$
and using \eqref{eq:poss}, produces
\begin{equation}
\frac{1}{N\Delta}(\psi_{N}-\psi_{0})=\frac{1}{N}\sum_{n=0}^{N-1}(\phi_{n}%
-\bar{\phi})+\frac{1}{N}\sum_{n=0}^{N-1}(\mathcal{L}^{\Delta}-\mathcal{L}%
)\psi_{n}+\frac{1}{N\Delta}\sum_{i=1}^{3}\big(M_{i,N}+S_{i,N}\big),
\label{sumedUp}%
\end{equation}
where
\begin{align*}
M_{1,N}  &  =\sum_{n=0}^{N-1}r_{n+1},\ M_{2,N}={\Delta^{\frac{1}{2}}}%
\sum_{n=0}^{N-1}D\psi_{n}[G_{n}\eta_{n+1}],\ M_{3,N}={\Delta^{\frac{3}{2}}%
}\sum_{n=0}^{N-1}D^{2}\psi_{n}[F_{n},G_{n}\eta_{n+1}],\\
S_{1,N}  &  =\frac{\Delta^{2}}{2}\sum_{n=0}^{N-1}D^{2}\psi_{n}[F_{n}%
,F_{n}],\ \ \ S_{2,N}=\sum_{n=0}^{N-1}R_{n+1},\ \,S_{3,N}=\frac{1}{6}%
\sum_{n=0}^{N-1}D^{3}\psi_{n}[\bar{\delta}_{n},\bar{\delta}_{n},\bar{\delta
}_{n}]\ .
\end{align*}
We will find it convenient to further decompose
\[
S_{3,N}=M_{0,N}+S_{0,N},
\]
where
\[
M_{0,N}={\Delta^{\frac{3}{2}}}\sum_{n=0}^{N-1}\Bigl(D^{3}\psi_{n}[G_{n}%
\eta_{n+1},G_{n}\eta_{n+1},G_{n}\eta_{n+1}]+3\Delta D^{3}\psi_{n}[F_{n}%
,F_{n},G_{n}\eta_{n+1}]\Bigr)\,.
\]
Notice that $\mathbf{E}[r_{n+1}|\mathcal{F}_{t_{n}}]=0$ and $\mathbf{E}%
[\eta_{n+1}|\mathcal{F}_{t_{n}}]=0$ and $\mathbf{E}[\eta_{n+1,i}\ \eta
_{n+1,j}\ \eta_{n+1,k}|\mathcal{F}_{t_{n}}]=0$. Then it is not difficult to
see that $M_{i,k},$ $i=0,\ldots,3,$ are martingales with respect to
$\{\mathcal{F}_{t_{k}}\}$ and, in particular, $\mathbf{E}M_{i,k}=0$ for any
$k$.

Since $f$ and $g$ are uniformly bounded, \eqref{goodMethod} implies that
$F_{n}$ and $G_{n}$ are uniformly bounded in $n$. Recall that we also know
that $\psi$ and its first four derivatives are uniformly bounded. Hence the
$S_{i,N}$ are bounded as follows:
\begin{align}
|S_{1,N}|  &  \leq{\Delta^{2}}\sum_{n=0}^{N-1}C_{1}|\phi|_{\infty}=C_{1}%
|\phi|_{\infty}\Delta T,\ \ \mathbf{E}|S_{2,N}|\leq\sum_{n=0}^{N-1}%
\mathbf{E}|R_{n+1}|\leq C_{2}|D^{2}\phi|_{\infty}\Delta T,\label{eq:iest}\\
\mathbf{E}|S_{0,N}|  &  \leq{\Delta^{2}}\sum_{n=0}^{N-1}C_{0}|D\phi|_{\infty
}=C_{0}|D\phi|_{\infty}\Delta T\,,\nonumber
\end{align}
where the $C_{i}$ are positive constants which we have labeled for future
reference. Similarly, we have (cf. (\ref{eq:genapp})):
\[
\Big|\sum_{n=0}^{N-1}(\mathcal{L}^{\Delta}-\mathcal{L})\psi_{n}\Big|\leq
\sum_{n=0}^{N-1}C_{4}|\phi|_{\infty}\Delta=C_{4}|\phi|_{\infty}\Delta N\,.
\]
Notice that this bound and the above bound in $S_{1,N}$ are a.s. bounds with
deterministic constants. Lastly observe that $|\psi_{N}-\psi_{0}|\leq
2|\phi|_{\infty}$. Applying all of the preceding estimates to \eqref{sumedUp}
produces the quoted result.
\end{proof}

\bigskip

Theorem~\ref{thm:veryBasicConv} implies that (cf. (\ref{eq:bias1})):
\begin{equation}
\mathrm{Bias}(\hat{\phi}_{N})=O\Big( \Delta+\frac{1}{T}\Big) .
\label{eq:bias2}%
\end{equation}

We now consider an $L^{2}$ convergence result, related to the mean ergodic
theorem (\ref{exampleE}) in the continuous case.

\begin{theorem}
\label{thm:basicConv} In the setting of Theorem~\ref{thm:veryBasicConv}, for
any $\phi\in H$, one has
\begin{equation}
\mathbf{E}\left(  \hat{\phi}_{N}-\bar{\phi}\right)  ^{2}\leq C\Big( \Delta
^{2}+\frac{1}{T}\Big) ,
\end{equation}
where $T=N\Delta$ and $C$ is some positive constant independent of $\Delta$
and $T$. Furthermore, $C$ depends on $\phi$ only through $|\phi|_{H}$ and does
so linearly.
\end{theorem}

\begin{proof}
As in the proof of Theorem~\ref{thm:veryBasicConv}, we provide details in the
elliptic case; in the hypoelliptic case the only change is that higher
derivatives of $\phi$ appear in the constants. We begin with \eqref{sumedUp}
from the proof of Theorem~\ref{thm:veryBasicConv} and obtain
\begin{align}
\mathbf{E}\Big( \frac{1}{N}\sum_{n=0}^{N-1}(\phi_{n}-\bar{\phi})\Big) ^{2}  &
\leq C\mathbf{E}\bigg\{ \frac{(\psi_{N}-\psi_{0})^{2}}{T^{2}}+\frac{1}{N^{2}%
}\Big( \sum_{n=0}^{N-1}(\mathcal{L}^{\Delta}-\mathcal{L})\psi_{n}%
\Big) ^{2}\label{eq:squar}\\
&  +\frac{1}{T^{2}}\sum_{i=0}^{2}S_{i,N}^{2}+\frac{1}{T^{2}}\sum_{i=0}%
^{3}M_{i,N}^{2}\bigg\} \ ,\nonumber
\end{align}
where $C$ is an ever changing constant. Observe that it follows from
(\ref{eq:iest}) that $|S_{1,N}|^{2}\leq C_{1}^{2}|\phi|_{\infty}^{2}\Delta
^{2}T^{2}.$ Further, by reasoning similar to that used in getting
(\ref{eq:iest}) we have
\begin{align*}
\mathbf{E}\,S_{2,N}^{2}  &  \leq\sum_{n,k=1}^{N}\mathbf{E}|R_{n}||R_{k}|\leq
C|D^{2}\phi|_{\infty}^{2}\Delta^{2}T^{2}\ ,\ \\
\mathbf{E}\,S_{0,N}^{2}  &  \leq{\Delta^{4}}\sum_{k,n=0}^{N-1}C|D\phi
|_{\infty}^{2}\!=C|D\phi|_{\infty}^{2}\Delta^{2}T^{2}\ .
\end{align*}
Since $M_{1,N}$ is a martingale, we get
\[
\mathbf{E}\,M_{1,N}^{2}=\sum_{n=0}^{N-1}\mathbf{E}\,r_{n+1}^{2}\leq\sum
_{n=0}^{N-1}C|\phi|_{\infty}^{2}\Delta^{2}\leq C|\phi|_{\infty}^{2}\Delta
^{2}N=C|\phi|_{\infty}^{2}\Delta T\;.
\]
Similar reasoning and the bound on $\mathbf{E}\eta_{i,n}^{8}$ give
\begin{align*}
\mathbf{E}\,M_{2,N}^{2}  &  \leq{\Delta}\sum_{n=0}^{N-1}C|\phi|_{\infty}%
^{2}\!=C|\phi|_{\infty}^{2}T,\ \ \ \ \ \ \mathbf{E}\,M_{3,N}^{2}\leq
{\Delta^{3}}\sum_{n=0}^{N-1}C|\phi|_{\infty}^{2}\!=C|\phi|_{\infty}^{2}%
\Delta^{2}T,\\
\mathbf{E}\,M_{0,N}^{2}  &  \leq{\Delta^{3}}\sum_{n=0}^{N-1}C|D\phi|_{\infty
}^{2}\!=C|D\phi|_{\infty}^{2}{\Delta^{2}}T\;.
\end{align*}
Using the bounds on $|\sum_{n=0}^{N-1}(\mathcal{L}^{\Delta}-\mathcal{L}%
)\psi_{n}|$ and $|\psi_{N}-\psi_{0}|$ from the proof of
Theorem~\ref{thm:veryBasicConv} and all the above inequalities, we estimate
the right-hand side of (\ref{eq:squar}) and arrive at the quoted result.
\end{proof}

\bigskip

We now prove an almost sure version of the preceding two results. Its relation
to Theorem~\ref{thm:basicConv} is the same as the relation of \eqref{eq:ascon}
to \eqref{exampleE}.

\begin{theorem}
\label{thm:moreConv}In the setting of Theorem~\ref{thm:veryBasicConv}, fixing
an $L>0$ there exists a deterministic constant $K$ (depending lineally on $L$)
so that for all $\phi$ with $\Vert\phi\Vert_{H}\leq L$, $\Delta$ sufficiently
small, positive $\varepsilon>0$, and $T$ sufficiently large one has:
\begin{equation}
\Big\vert \hat{\phi}_{N}-\bar{\phi}\Big\vert \leq K\Delta+\frac{C(\omega
)}{T^{1/2-\varepsilon}}\ \ \ \ \text{a.s.,} \label{eq:mm}%
\end{equation}
where $T=N\Delta$ and $C(\omega)>0$ is an a.s. bounded random variable
depending on $\varepsilon$ and the particular $\phi$.
\end{theorem}

\begin{proof}
As in the proof of Theorem~\ref{thm:veryBasicConv}, we provide details in the
elliptic case; in the hypoelliptic case the only change is that higher
derivatives of $\phi$ appear in the constants. Starting from \eqref{sumedUp},
we have
\[
\Big\vert\frac{1}{N}\sum_{n=0}^{N-1}\phi_{n}-\bar{\phi}\Big\vert\leq
\frac{\left\vert \psi_{N}-\psi_{0}\right\vert }{T}+\frac{1}{N}\sum_{n=0}%
^{N-1}\left\vert (\mathcal{L}^{\Delta}-\mathcal{L})\psi_{n}\right\vert
+\frac{1}{T}\sum_{i=0}^{2}\left\vert S_{i,N}\right\vert +\frac{1}{T}\sum
_{i=0}^{3}|M_{i,N}|.
\]
Recall that (see the proof of Theorem~\ref{thm:veryBasicConv})%
\begin{align*}
\left\vert \psi_{N}-\psi_{0}\right\vert  &  \leq2|\phi|_{\infty}%
,\ \ \sum_{n=0}^{N-1}\left\vert (\mathcal{L}^{\Delta}-\mathcal{L})\psi
_{n}\right\vert \leq C_{4}|\phi|_{\infty}\ \Delta N,\ \ \left\vert
S_{1,N}\right\vert \leq C_{1}|\phi|_{\infty}\ \Delta T,\ \\
\left\vert S_{2,N}\right\vert  &  \leq K\Delta^{2}\sum_{n=0}^{N-1}\left\vert
\eta_{n+1}\right\vert ^{4}+K\Delta^{4}N,\ \ \left\vert S_{0,N}\right\vert \leq
K\Delta^{2}\sum_{n=0}^{N-1}\left\vert \eta_{n+1}\right\vert ^{2}+K\Delta
^{3}N,\
\end{align*}
where $K$ is an ever changing positive deterministic constant, independent of
$\Delta$ and $N.$
%AMS added to previous sentence
Due to the strong law of large numbers, the sum $\frac{1}{N}\sum_{n=0}%
^{N-1}\left\vert \eta_{n+1}\right\vert ^{4}$ a.s. converges\footnote{Note that
in the case of $\eta_{n,i}$ from (\ref{eq:dis}) this sum is equal to $m^{2}$.}
to $\mathbf{E}\left\vert \eta_{1}\right\vert ^{4}$ and thus for almost every
sequence $\eta_{1},$ $\eta_{2},$ $\ldots$ and all sufficiently large $N$ we
have $\frac{1}{N}\sum_{n=0}^{N-1}\left\vert \eta_{n+1}\right\vert ^{4}<K$ for
some deterministic $K>0.$ Hence
\begin{equation}
\Big\vert\frac{1}{N}\sum_{n=0}^{N-1}\phi_{n}-\bar{\phi}\Big\vert\leq
K\Big(\Delta+\frac{1}{T}\Big)+\Theta_{N}\ \ \ \ \ \ \text{a.s. },
\label{eq:sta}%
\end{equation}
where $\Theta_{N}\overset{\mbox{\tiny def}}{=}\frac{1}{T}\sum_{i=0}%
^{3}|M_{i,N}|.$

Now we analyze $\Theta_{N}.$ We have for $r\geq1$%
\[
\mathbf{E}\Theta_{N}^{2r}\leq\frac{K}{T^{2r}}\sum_{i=0}^{3}\mathbf{E}%
|M_{i,N}|^{2r}%
\]
(here $K$ depends on $r$). Recall that $M_{i,k}$ are martingales. If we can
prove that $\mathbf{E}|M_{i,N}|^{2r}\leq CN^{r}$ for all $r$ sufficiently
large then the proof can be completed with the aid of the Borel-Cantelli lemma
as in Section~\ref{sec:LLN}.

We will the provide argument for estimating $\mathbf{E}|M_{2,N}|^{2r},$ the
other terms are estimated analogously. We re-write
\[
M_{2,N}=\sqrt{\Delta}\sum_{n=0}^{N-1}D\psi_{n}[G_{n}\eta_{n+1}%
]\overset{\mbox{\tiny def}}{=}\sqrt{\Delta}\tilde{M}_{2,N}\ .
\]
We have $\tilde{M}_{2,1}=0$ and for an integer $r>0:$
\begin{align}
\mathbf{E}\tilde{M}_{2,k+1}^{2r}  &  =\mathbf{E}\left(  \tilde{M}_{2,k}%
+D\psi_{k}[G_{k}\eta_{k+1}]\right)  ^{2r}\label{eq:m1a}\\
&  =\mathbf{E}\left(  \tilde{M}_{2,k}^{2}+2\tilde{M}_{2,k}\ D\psi_{k}%
[G_{k}\eta_{k+1}]+\left(  D\psi_{k}[G_{k}\eta_{k+1}]\right)  ^{2}\right)
^{r}\nonumber\\
&  \leq\mathbf{E}\tilde{M}_{2,k}^{2r}+2r\mathbf{E}\left(  \tilde{M}%
_{2,k}^{2r-1}\ D\psi_{k}[G_{k}\eta_{k+1}]\right) \nonumber\\
&  +K\sum_{l=2}^{2r}\mathbf{E}\left(  |\tilde{M}_{2,k}|^{2r-l}|D\psi_{k}%
[G_{k}\eta_{k+1}]|^{l}\right)  .\nonumber
\end{align}
Note that the second term is equal to zero. Indeed
\begin{equation}
\mathbf{E}\left(  \tilde{M}_{2,k}^{2r-1}D\psi_{k}[G_{k}\eta_{k+1}]\right)
=\mathbf{E}\left(  \tilde{M}_{2,k}^{2r-1}\mathbf{E}\left(  D\psi_{k}[G_{k}%
\eta_{k+1}]\ |\mathcal{F}_{t_{k}}\right)  \right)  =0\ . \label{eq:m1b}%
\end{equation}
Using the elementary inequality
\[
ab=\frac{1}{N}\left(  abN\right)  \leq\frac{1}{N}\left(  \dfrac{a^{p}}%
{p}+\dfrac{b^{q}N^{q}}{q}\right)  ,\ a,b>0,\ \ p,q>1,\ \dfrac{1}{p}+\dfrac
{1}{q}=1,
\]
we get
\begin{align}
\mathbf{E}\left(  |\tilde{M}_{2,k}|^{2r-l}|D\psi_{k}[G_{k}\eta_{k+1}%
]|^{l}\right)   &  \leq\frac{1}{N}\mathbf{E}\left(  \frac{2r-l}{2r}|\tilde
{M}_{2,k}|^{2r}+\frac{l}{2r}|D\psi_{k}[G_{k}\eta_{k+1}]|^{2r}N^{2r/l}\right)
,\label{eq:m1c}\\
l  &  =2,\ldots,2r\ .\nonumber
\end{align}
%AMS I changed 2r-1 into 2r in the previous line
The relations (\ref{eq:m1a})-(\ref{eq:m1c}) imply
\[
\mathbf{E}\tilde{M}_{2,k+1}^{2r}\leq\mathbf{E}\tilde{M}_{2,k}^{2r}(1+\frac
{K}{N})+KN^{r-1}\ ,
\]
whence
\[
\mathbf{E}\tilde{M}_{2,N}^{2r}\leq KN^{r}\ .
\]
Note that by Jensen's inequality this inequality holds for non-integer
$r\geq1$ as well. Therefore,
\[
\frac{1}{T^{2r}}\mathbf{E}\left\vert M_{2,N}\right\vert ^{2r}=\frac{\Delta
^{r}}{T^{2r}}\mathbf{E}\left\vert \tilde{M}_{2,N}\right\vert ^{2r}\leq\frac
{K}{T^{r}}\ .
\]
Analogously, we obtain
\[
\frac{1}{T^{2r}}\mathbf{E}\left\vert M_{1,N}\right\vert ^{2r}\leq\Delta
^{r}\frac{K}{T^{r}},\ \ \frac{1}{T^{2r}}\mathbf{E}\left\vert M_{3,N}%
\right\vert ^{2r}\leq\Delta^{2r}\frac{K}{T^{r}},\ \ \frac{1}{T^{2r}}%
\mathbf{E}\left\vert M_{0,N}\right\vert ^{2r}\leq\Delta^{2r}\frac{K}{T^{r}%
}\ .
\]
Thus,
\begin{equation}
\mathbf{E}\Theta_{N}^{2r}\leq\frac{K}{T^{r}}. \label{eq:the}%
\end{equation}
The Markov inequality together with (\ref{eq:the}) implies
\[
P\Big(\Theta_{N}>\frac{1}{\Delta^{\gamma}N^{\gamma}}\Big)\leq\Delta^{2r\gamma
}N^{2r\gamma} (\mathbf{E}\Theta_{N}^{2r})\leq KT^{r(2\gamma-1)}.
\]
Then for any $\gamma=1/2-\varepsilon$, with $\varepsilon>0$, there is a
sufficiently large $r\geq1$ such that (recall that $\Delta$ is fixed here and
$T=\Delta N)$
\[
\sum_{N=1}^{\infty}P\left(  \Theta_{N}>\frac{1}{T^{\gamma}}\right)  \leq
K\sum_{N=1}^{\infty}T^{r(2\gamma-1)}<\infty.
\]
Hence, by the Borel-Cantelli lemma, the random variable $\varsigma
\overset{\mbox{\tiny def}}{=}\sup_{T>0}T^{\gamma}\Theta_{N}\ $ is a.s. finite
which together with (\ref{eq:sta}) implies (\ref{eq:mm}).
\end{proof}

We note that Theorems~\ref{thm:veryBasicConv}-\ref{thm:moreConv} do not
require the Markov chain $X_{n}$ to be ergodic. It is, of course, possible
under some additional conditions on the numerical method to prove its
ergodicity (see for example
\cite{Talay90,TalayHam02,MattinglyStuartHighamSDENUM02}). If the limit
$\lim_{N\rightarrow\infty}\hat{\phi}_{N}$ exists and is independent of $X_{0}$
(i.e., if the Markov chain $X_{n}$ is ergodic) then it follows from
Theorem~\ref{thm:moreConv} that $\lim_{N\rightarrow\infty}\big\vert\hat{\phi
}_{N}-\bar{\phi}\big\vert\leq K\Delta\ \ $a.s., which is consistent with the
results of, for example, \cite{Talay90}.

\begin{remark}
In the series of papers
\cite{LambertonPages02,LambertonPages03RCIM,Lem05,Lem07,Pan08a,Pan08b} the
authors consider the following weighted estimator for the stationary average
$\bar{\phi}:$
\[
\tilde{\phi}_{N}=\frac{\sum_{n=1}^{N}w_{n}\phi(X_{n})}{\sum_{n=1}^{N}w_{n}}\ ,
\]
where $X_{n}$ are obtained by an Euler-type scheme with a decreasing time step
$\Delta_{n}$ such that $\Delta_{n}\rightarrow0$ as $n\rightarrow\infty$ while
positive weights $w_{n}$ are such that $\sum_{n=1}^{N}w_{n}\rightarrow\infty$
as $N\rightarrow\infty.$ In particular, they proved that for some Euler-type
schemes and $w_{n}=\Delta_{n}=\Delta_{0}n^{-\alpha},$ $\alpha\in
\lbrack1/3,1),$
\begin{align*}
\mathrm{Bias}\left(  \tilde{\phi}_{N}\right)   &  =\left\{
\begin{array}
[c]{cc}%
0, & \alpha\in(1/3,1)\\
O(1/N^{1/3}), & \alpha=1/3,
\end{array}
\right.  \\
\mathrm{Var}\left(  \tilde{\phi}_{N}\right)   &  =O\left(  \frac
{1}{N^{1-\alpha}}\right)  \ ,
\end{align*}
i.e., roughly speaking, the error of the estimator $\tilde{\phi}_{N}$ under
the optimal choice of the parameters is
\[
\tilde{\phi}_{N}-\bar{\phi}\sim O\left(  \frac{1}{N^{1/3}}\right)  \ .
\]
We also note that the authors of \cite{LambertonPages02,LambertonPages03RCIM}
proved a.s. convergence of weighted empirical measures based on Euler-type
schemes with decreasing step to the invariant measure of the corresponding SDE
exploiting the Echeverria-Weiss theorem, which differs from the approach used
in this paper.

Let us briefly compare the estimators $\tilde{\phi}_{N}$ and $\hat{\phi}_{N}$.
For the estimator $\hat{\phi}_{N}$ from (\ref{eq:estim}) with equal weights
$w_{n}=1$ and with $X_{n}$ obtained by an Euler-type scheme with constant time
step $\Delta$, we can say that the error is (cf. Theorems~\ref{thm:basicConv}
and~\ref{thm:moreConv}):
\begin{equation}
\hat{\phi}_{N}-\bar{\phi}\sim O\left(  \Delta+\frac{1}{N^{1/2}\Delta^{1/2}%
}\right)  \ .\label{eq:rough}%
\end{equation}
If we fix the computational costs (i.e., $N$ and an Euler-type scheme) then
the asymptotically optimal choice of $\Delta$ for $\hat{\phi}_{N}$ is
$N^{-1/3}$ in which case $O\left(  \Delta+1/N^{1/2}\Delta^{1/2}\right)
=O\left(  1/N^{1/3}\right)  .$ It is interesting to note that these optimal
orders of errors of both estimators are the same. At the same time, we should
emphasize that the term with $\Delta$ in (\ref{eq:rough}) (see also $K\Delta$
(\ref{eq:bias2}) and (\ref{eq:mm})) is related to the numerical integration
error while the term with $1/\sqrt{T}=1/N^{1/2}\Delta^{1/2}$ is related to the
statistical error (see also Section~\ref{sec:practice}). In the case of large
scale simulations (like those in molecular dynamics) the numerical error is
usually relatively small thanks to the existing state of the art numerical
integrators and the statistical error prevails. In such common in practice
situations one chooses $\Delta$ and $N$ to appropriately control the
corresponding errors (see further discussion in Section~\ref{sec:practice}).
\end{remark}

\subsection{Higher order schemes\label{sec:high}}

We now consider a more general approximation than \eqref{eq:euler}. In
addition to the assumptions made in Section~\ref{sec:sde} we assume in this
section that the coefficients of the SDE \eqref{eq:sde} and the function
$\phi$ are sufficiently smooth.

Given a function $\bar{\delta}:\mathbf{T}^{d}\times(0,1)\times\mathbf{R}%
^{m}\rightarrow\mathbf{T}^{d}$ and a sequence of $\mathbf{R}^{m}$-valued
i.i.d. random variables $\{\xi_{n}=(\xi_{n,1},\ldots,\xi_{n,m}):n\in
\mathbf{N}\}$, we define a general numerical method
\begin{equation}
\left\{
\begin{aligned} X_{n+1} &= X_n + \bar \delta(X_n,\Delta,\xi_{n+1})\, , \\ X_0&=x\, . \end{aligned}\right.
\label{eq:generalMethod}%
\end{equation}
We assume that the $\xi_{n}$ have sufficiently high moments finite. Clearly
our previous class of methods fits in to this framework as well as a number of
new methods such as implicit Euler. Guided by (\ref{locer}) we make the
following general assumption about \eqref{eq:generalMethod} after which we
will state some easier to verify conditions which are equivalent.

\vspace{1ex}

\begin{assumption}
\label{a:localError2} For all $\Delta\in(0,1)$ sufficiently small, and all
$\phi\in W^{2(p+1),\infty}$ if $X_{0}=X(0)$ then
\begin{equation}
\mathbf{E}\,\phi(X_{1})-\mathbf{E}\,\phi(X(\Delta))=O(\Delta^{p+1}),
\end{equation}
where the constant in the error term is uniform over all $\phi$ with
$\Vert\phi\Vert_{W^{2(p+1),\infty}}\leq1$.
\end{assumption}

\vspace{1ex}

To complement the increments of the numerical method $\bar{\delta}$ defined
above we now define the $\delta$ increments of the SDE. Namely for any
$x\in\mathbf{T}^{d}$, we define $\delta:\mathbf{T}^{d}\times(0,1)\times
\Omega\rightarrow\mathbf{T}^{d}$ by
\begin{equation}
\delta(x,\Delta;\omega)\overset{\mbox{\tiny def}}{=}X(\Delta;\omega
)-x,\label{eq:delta}%
\end{equation}
where $X(0;\omega)=x$ and $X(t;\omega)$ solves \eqref{eq:sde}. The following
proposition, whose proof is given at the end of the section, enables
Assumption \ref{a:localError2} to be verified.

\begin{proposition}
\label{prop:weakOrderConditions} Assume that for some $p\in\mathbf{N}$, there
exists a positive constant $K$ so that for all $\Delta\in(0,1)$ sufficiently
small:
\begin{gather}
\sup_{\substack{(\alpha_{1},\dots,\alpha_{s})\\\alpha_{i}\in\{1,\dots
,d\}}}\Big|\mathbf{E}\big(\prod_{i=1}^{s}\delta_{\alpha_{i}}-\prod_{i=1}%
^{s}\bar{\delta}_{\alpha_{i}}\big)\Big|\leq K\Delta^{p+1},\;s=1,\ldots
,2p+1,\;\label{Db03}\\
\sup_{\substack{(\alpha_{1},\dots,\alpha_{2p+2})\\\alpha_{i}\in\{1,\dots
,d\}}}\mathbf{E}\prod_{i=1}^{2p+2}\big|\bar{\delta}_{\alpha_{i}}\big|\leq
K\Delta^{p+1}\,. \label{Db04}%
\end{gather}
Then the method satisfies Assumption~\ref{a:localError2} with the same $p$.
\end{proposition}

We note that examples of second-order weak schemes ($p=2$) which satisfy
Assumption~\ref{a:localError2} can be found in many places including \cite[p.
103]{MilsteinTretyakov2004} and \cite{AndersonMattingly09}. Higher order
methods also exist \cite{MilsteinTretyakov2004}. We also note that it follows
from the general theory \cite[p. 100]{MilsteinTretyakov2004} that the
numerical schemes considered in Proposition~\ref{prop:weakOrderConditions}
have weak convergence of order $p$ on finite time intervals. Now we prove a
mean convergence result, which is analogous to Theorem~\ref{thm:veryBasicConv}
in the case of Euler-type methods (\ref{eq:euler}).

\begin{theorem}
\label{thm:high} Let Assumptions~\ref{ass:standing} and \ref{a:localError2}
hold. Let $H$ denote $W^{2p,\infty}$ in the elliptic setting and
$W^{2(p+1),\infty}$ in the hypoelliptic setting. Then for any $\phi\in H$ with
$|\phi|_{H}\leq1$, consider $\hat{\phi}_{N}$ defined by \eqref{eq:estim} where
$X_{n}$ is generated by the numerical method from \eqref{eq:generalMethod}
rather than \eqref{eq:euler} as previously. Then
\begin{equation}
\Big|\mathbf{E}\hat{\phi}_{N}-\bar{\phi}\Big|\leq C\Big(\Delta^{p}+\frac{1}%
{T}\Big), \label{eq:m3}%
\end{equation}
where $T=N\Delta$ and $C$ is some positive constant independent of $\Delta$
and $T$. Furthermore, the constant $C$ is a linear function of $|\phi|_{H}$
and otherwise independent of $\phi$. In other words,
\begin{equation}
\mathrm{Bias}(\hat{\phi}_{N})=O\Big(\Delta^{p}+\frac{1}{T}\Big)\ .
\label{eq:bias3}%
\end{equation}

\end{theorem}

\begin{proof}
(of Theorem \ref{thm:high}) As before, let $\psi$ be the solution to the Poisson equation
associated to $\phi$ given in \eqref{eq:poss}. Define $\psi_{n}=\psi(X_{n})$
and let $X(x,t)$ be the solution to \eqref{eq:sde} at time $t$ with initial
condition $x$. From our assumptions, we have that
\begin{equation}
\mathbf{E}\psi_{n+1}=\mathbf{E}\psi(X(X_{n},\Delta))+O(\Delta^{p+1}).
\label{ta12}%
\end{equation}
Rewriting $\mathbf{E}\psi(X(X_{n},\Delta))$ in \eqref{ta12} via the Taylor
expansion of expectations of SDE solution (\cite[Lemma 2.1.9, p. 99]%
{MilsteinTretyakov2004} or \cite{KloedenPlaten92}), we arrive at
\begin{equation}
\mathbf{E}\psi_{n+1}=\mathbf{E}\psi_{n}+\sum_{k=1}^{p}\frac{\Delta^{k}}%
{k!}\mathbf{E}\left(  \mathcal{L}^{k}\psi\right)  (X_{n})+O(\Delta^{p+1}).
\label{ta2}%
\end{equation}
Summing \eqref{ta2} over the first $N$ terms and dividing by $N\Delta,$ we
obtain
\begin{equation}
\frac{\mathbf{E}\psi_{N}-\psi(x)}{T}=\frac{1}{N}\sum_{n=0}^{N-1}%
\mathbf{E}\left(  \mathcal{L}\psi\right)  (X_{n})+\sum_{k=2}^{p}\frac
{\Delta^{k-1}}{k!}Q_{k}+O(\Delta^{p}), \label{ta30}%
\end{equation}
where $Q_{k}=\frac{1}{N}\sum_{n=0}^{N-1}\mathbf{E}\left(  \mathcal{L}^{k}%
\psi\right)  (X_{n}).$ Using \eqref{eq:poss} and boundedness of $\psi$, we
have after rearrangement of (\ref{ta30}):
%AMS changed varphi to phi in next line%
\begin{equation}
\Big\vert\frac{1}{N}\sum_{n=0}^{N-1}\mathbf{E}\phi(X_{n})-\bar{\phi
}\Big\vert\leq\sum_{k=2}^{p}\frac{\Delta^{k-1}}{k!}\left\vert Q_{k}\right\vert
+K\Delta^{p}+\frac{K}{T}. \label{ta31}%
\end{equation}
Applying analogous arguments to $\mathcal{L}^{k-1}\psi,$ $k=2,\ldots,p,$ as we
did for $\psi$ in (\ref{ta30}), we have
\[
\frac{\mathbf{E}\mathcal{L}^{k-1}\psi(X_{N})-\mathcal{L}^{k-1}\psi(x)}{T}%
=\sum_{i=k}^{p}\frac{\Delta^{i-k}}{(i+1-k)!}Q_{i}+O(\Delta^{p+1-k}).
\]
Therefore (cf. (\ref{ta31})),
\begin{equation}
|Q_{k}|\leq\sum_{i=k+1}^{p}\frac{\Delta^{i-k}}{(i+1-k)!}\left\vert
Q_{i}\right\vert +K\Delta^{p+1-k}+\frac{K}{T},\ \ k=2,\ldots,p,\nonumber
\end{equation}
and, in particular, $|Q_{p}|\leq K\Delta+K/T.$ Hence $|Q_{k}|\leq
K\Delta^{p+1-k}+K/T$ which together with (\ref{ta31}) implies (\ref{eq:m3}%
).\footnote{To help with intuitive understanding of the proof, we remark that
$\int\mathcal{L}^{k}\psi(x)\,d\mu(x)=0$ which is approximated by $Q_{k}$ with
sufficient accuracy.}
\end{proof}

\begin{remark}
\label{rem:genm2} If we substitute $X_{n}$ from the method
\eqref{eq:generalMethod} in $\hat{\phi}_{N}$ from \eqref{eq:estim}, it is also
possible to prove (analogous to Theorem~\ref{thm:moreConv}) that
\begin{equation}
\left\vert \hat{\phi}_{N}-\bar{\phi}\right\vert \leq K\Delta^{p}%
+\frac{C(\omega)}{T^{1/2-\varepsilon}}\ \ \ \ \text{a.s.\ \ .} \label{eq:mm2}%
\end{equation}

\end{remark}

\begin{proof}
(of Proposition \ref{prop:weakOrderConditions}) We note that under the
assumptions made the solution $\psi$ of the Poisson equation \eqref{eq:poss}
is sufficiently smooth. Expanding $\psi_{n+1} =\psi(X_{n+1})$ in powers of
$\bar{\delta}_{n}\overset{\mbox{\tiny def}}{=} X_{n+1}-X_{n}$ and taking
expectation, we get
\begin{align}
\mathbf{E}\psi_{n+1}  &  =\mathbf{E}\psi_{n}+\mathbf{E}D\psi_{n}[\bar{\delta
}_{n}]+\frac{1}{2}\mathbf{E}D^{2}\psi_{n}[\bar{\delta}_{n},\bar{\delta}
_{n}]+\cdots\label{ta1}\\
&  +\frac{1}{\left(  2p+1\right)  !}\mathbf{E}D^{2p+1}\psi_{n}%
[\underset{2p+1}{\underbrace{\bar{\delta}_{n},\ldots,\bar{\delta}_{n}}%
}]+O(\Delta^{p+1}),\nonumber
\end{align}
where the remainder $|O(\Delta^{p+1})|\leq K\Delta^{p+1}$ with $K$ independent
of $\Delta$ and $n$. If we replace $\bar{\delta}_{n}$ by $\delta_{n}$ defined
in \eqref{eq:delta} in (\ref{ta1}) then, using the conditional version of
(\ref{Db03}), we obtain (with a different remainder $O(\Delta^{p+1})$ than in
(\ref{ta1})):
\begin{align}
\mathbf{E}\psi_{n+1}  &  =\mathbf{E}\psi_{n}+\mathbf{E}D\psi_{n}[\delta
_{n}]+\frac{1}{2}\mathbf{E}D^{2}\psi_{n}[\delta_{n},\delta_{n}]+\cdots
\label{ta11}\\
&  +\frac{1}{\left(  2p+1\right)  !}\mathbf{E}D^{2p+1}\psi_{n}%
[\underset{2p+1}{\underbrace{\delta_{n},\ldots,\delta_{n}}}]+O(\Delta
^{p+1}).\nonumber
\end{align}
It is not difficult to see that the right-hand side of \eqref{ta11} coincide
with expectation of the Taylor expansion of $\psi(X(X_{n},\Delta))$ around
$X_{n}$ up to a reminder of order $\Delta^{p+1}.$ Hence
\[
\mathbf{E}\psi_{n+1}=\mathbf{E}\psi((X(X_{n},\Delta))+O(\Delta^{p+1})
\]
and the proof is complete.
\end{proof}

%\begin{remark}
%\label{rem:genm}Conditions (\ref{Db03})-(\ref{Db04}) are equivalent to (cf.
%Remark~\ref{rmk:localError} and also (\ref{ta12})):
%\[
%\mathbf{E}\phi(X(\Delta))-\mathbf{E}\phi(X_{1})=O(\Delta^{p+1}).
%\]
%\end{remark}

\subsection{The Richardson-Romberg (Talay-Tubaro) error
expansion\label{sec:expan}}

In this section we consider an expansion of the global error $\mathbf{E}%
\hat{\phi}_{N}-\bar{\phi}$ in powers of the time step $\Delta$ analogous to
the Talay-Tubaro result \cite{TalayTubaro90}. As in the previous section, we
assume here that the coefficients of the SDE \eqref{eq:sde} and the function
$\phi$ are sufficiently smooth. Note that we obtain the expansion both in the
elliptic and hypoelliptic setting.

\begin{theorem}
\label{thm:expan}Let Assumptions~\ref{ass:standing} and \ref{a:localError2}
hold. Let $q$ be a positive integer and $H$ denote $W^{2(p+q+1),\infty}$ in
the elliptic setting and $W^{2(p+q+2),\infty}$ in the hypoelliptic setting.
For any $\phi\in H$ with $|\phi|_{H}\leq1$, consider $\hat{\phi}_{N}$ defined
by \eqref{eq:estim} where $X_{n}$ is generated by the numerical method from
\eqref{eq:generalMethod}. Then
\begin{equation}
\mathbf{E}\hat{\phi}_{N}-\bar{\phi}=C_{0}^{p}\Delta^{p}+\cdots+C_{q}%
\Delta^{p+q}+O\left(  \Delta^{p+q+1}+\frac{1}{T}\right)  \ , \label{exp1}%
\end{equation}
where $T=N\Delta;$
\[
\left\vert O\left(  \Delta^{p+q+1}+\frac{1}{T}\right)  \right\vert \leq
K\left(  \Delta^{p+q+1}+\frac{1}{T}\right)  \ ;
\]
and the constants $C_{0},\ldots,C_{q},$ $K$ are independent of $\Delta$ and
$T$ and they are linear function of $|\phi|_{H}$ and otherwise independent of
$\phi$. In other words,
\[
\mathrm{Bias}(\hat{\phi}_{N})=C_{0}^{p}\Delta^{p}+\cdots+C_{q}\Delta
^{p+q}+O\left(  \Delta^{p+q+1}+\frac{1}{T}\right)  \ .
\]

\end{theorem}

\begin{proof}
Here we make use of the Poisson equation again and also exploit an idea used
in the proof of the Talay-Tubaro expansion in the finite time case from
\cite[pp. 106-108]{MilsteinTretyakov2004}. We prove the theorem in the case of
$p=1$ (i.e., for Euler-type schemes) and $q=0$ for the clarity of the
exposition. It is not difficult to extend the proof to arbitrary $p,q>0.$

In the case of $p=1$ for a particular Euler-type scheme, we can write (cf.
(\ref{ta2})):
\begin{equation}
\mathbf{E}\psi_{n+1}=\mathbf{E}\psi_{n}+\Delta\mathbf{E}\left(  \mathcal{L}%
\psi\right)  (X_{n})+\Delta^{2}\mathbf{E}A(X_{n})\mathbf{+}O(\Delta^{3})\ ,
\label{exp2}%
\end{equation}
where $A(x)$ is the coefficient at $\Delta^{2}$ in the corresponding expansion
of $\mathbf{E}\psi_{1}$ at $X_{0}=x.$ For instance, in the case of the
explicit Euler-Maruyama scheme (\ref{Euler-Mar})%
\[
A(x)=\frac{1}{2}\sum_{i,j=1}^{d}f_{i}f_{j}\frac{\partial^{2}\psi}{\partial
x_{i}\partial x_{j}}+\frac{1}{2}\sum_{i,j,k=1}^{d}f_{i}a_{jk}\frac
{\partial^{3}\psi}{\partial x_{i}\partial x_{j}\partial x_{k}}+\frac{1}%
{24}\sum_{i,j,k,l=1}^{d}a_{ij}a_{kl}\frac{\partial^{4}\psi}{\partial
x_{i}\partial x_{j}\partial x_{k}\partial x_{l}},
\]
where all the coefficients and the derivatives of the function $\psi$ are
evaluated at $x.$ We do not need the explicit form of $A(x)$ for the proof.

Summing (\ref{exp2}) over the first $N$ terms, dividing by $N\Delta,$ using
(\ref{eq:poss}) and rearranging the terms, we obtain
\begin{align}
\mathbf{E}\hat{\phi}_{N}-\bar{\phi}  &  =\frac{1}{T}(\mathbf{E}\psi_{N}%
-\psi_{0})-\frac{\Delta}{N}\sum_{n=0}^{N-1}\mathbf{E}A(X_{n})\mathbf{+}%
O(\Delta^{2})\label{exp3}\\
&  =-\frac{\Delta}{N}\sum_{n=0}^{N-1}\mathbf{E}A(X_{n})\mathbf{+}O\left(
\Delta^{2}+\frac{1}{T}\right)  \ .\nonumber
\end{align}
Due to Theorem~\ref{thm:veryBasicConv}, we have
\begin{equation}
\mathbf{E}\frac{1}{N}\sum_{n=0}^{N-1}A(X_{n})-\bar{A}=O\left(  \Delta+\frac
{1}{T}\right)  \,, \label{exp4}%
\end{equation}
where the constant $\bar{A}$ is the stationary average of $A$ (see
(\ref{eq:ergli})). The expansion (\ref{exp1}) with $p=1,$ $q=0$ and
$C_{0}=-\bar{A}$ follows from (\ref{exp3}) and (\ref{exp4}).
\end{proof}

\section{Error analysis for the numerical stationary measures}

\label{sec:ClosenessII}

\subsection{Distances between true and approximate stationary measures}

We now use the results of the previous section to prove that any stationary
measure of the numerical method is close to that of the underlining SDE. The
existence of stationary measures for the numerical method follows by the
Krylov-Bogoliubov construction.\footnote{The fact that our state space is
compact ensures that the time-averaged transition measure forms a
tight family of probability measures
\cite{Hasminskii1980}.} We have assumed in \eqref{eq:SDEexpand} and
\eqref{eq:genApprox} or Assumption~\ref{a:localError2} that the finite time
dynamics of our method and SDE are close. This can be seen as a form of
\textquotedblleft consistency\textquotedblright. It is reasonable to expect
that the longtime behavior will be close since our setting of the torus
provides the necessary \textquotedblleft stability\textquotedblright through ergodicity.

%jcm - we have and invariant measure and are assuming noting so I
%modified
%In this section we assume that the numerical method
%under consideration generates the Markov chain $X_{n}$ which has a stationary
%measure. This is very natural, since we have shown that the numerical time
%averages are close to the true time average on long time intervals. Since for
%bounded test functions, the time averages converge to the integral of the test
%function against the invariant measure this is enough to specify the measure.

We begin by giving a metric in which we will measure the distance between
measures. If $\mu^{\Delta}$ is a stationary measure of the numerical method,
then for any bounded function $\phi:\mathbf{T}^{d}\rightarrow\mathbf{R}$ and
$n\in\mathbf{N:}$
\begin{equation}
\int_{\mathbf{T}^{d}}\mathbf{E}_{z}\phi(X_{n})\mu^{\Delta}(dz)=\int%
_{\mathbf{T}^{d}}\phi(z)\mu^{\Delta}(dz),
\end{equation}
where $\mathbf{E}_{z}$ denotes the expectation conditional on $X(0)=z$. Fixing
an integer $p \geq1$, we define the metric $\rho$ between two probability
measures on $\mathbf{T}^{d} $ by
\[
\rho(\nu_{1},\nu_{2})=\sup_{\phi\in\mathcal{H}}\left(  \int\phi(z)\nu
_{1}(dz)-\int\phi(z)\nu_{2}(dz)\right)  ,
\]
where $\mathcal{H}=\{\phi:\mathbf{T}^{d}\rightarrow\mathbf{R}:|\phi|_{H}%
\leq1\}$ and $H=W^{2p,\infty}$ in the elliptic setting and $H=W^{2(p+1),\infty
}$ in the hypoelliptic setting. Observe that since $\phi\in\mathcal{H}$
implies $-\phi\in\mathcal{H}$, one also has the equivalent, and sightly more
standard, characterization of $\rho$ given by
\[
\rho(\nu_{1},\nu_{2})=\sup_{\phi\in\mathcal{H}}\left\vert \int\phi(z)\nu
_{1}(dz)-\int\phi(z)\nu_{2}(dz)\right\vert \,.
\]

%AMS slight reword of theorem

\begin{theorem}
\label{cor:invMeasureClose} Defining the metric $\rho$ as above for some
integer $p \geq1$ and assume that either:

\begin{enumerate}
\item $p=1$ and Assumptions~\ref{ass:standing} and \ref{ass:eular} hold; or

\item Assumptions~\ref{ass:standing} and \ref{a:localError2} hold.
\end{enumerate}

then there exists a positive $C$ so that if $\mu^{\Delta}$ is any stationary
measure of the numerical method \eqref{eq:euler} then
\begin{align*}
\rho(\mu^{\Delta}, \mu) \leq C\Delta^{p}\;.
\end{align*}

\end{theorem}

\begin{proof}
Since $\mu^{\Delta}$ is stationary, we have that $\int\phi(z)\mu^{\Delta
}(dz)=\int\mathbf{E}_{z}\phi(X_{n})\mu^{\Delta}(dz)$ for any $n\geq0$ and
hence
\[
\int\phi(z)\mu^{\Delta}(dz)=\int\frac{1}{N}\sum_{n=0}^{N-1}\mathbf{E}_{z}%
\phi(X_{n})\,\mu^{\Delta}(dz).
\]
Then
\begin{align*}
\int\phi(z)\mu^{\Delta}(dz)-\int\phi(z)\mu(dz) &  =\int\Big[\frac{1}{N}%
\sum_{n=0}^{N-1}\mathbf{E}_{z}\phi(X_{n})-\bar{\phi}\Big]\,\mu^{\Delta}(dz)\\
&  \leq\int\Big|\frac{1}{N}\sum_{n=0}^{N-1}\mathbf{E}_{z}\phi(X_{n})-\bar
{\phi}\Big|\,\mu^{\Delta}(dz)\leq C\Big(\Delta^{p}+\frac{1}{\Delta N}\Big)\ ,
\end{align*}
where in the last estimate we have invoked either
Theorem~\ref{thm:veryBasicConv} or Theorem~\ref{thm:high} depending on which
assumptions hold. Now taking $N\rightarrow\infty$, proves the result since the
right-hand side is uniform for any $\phi\in\mathcal{H}$.
\end{proof}

We reemphasize that we have not assumed in this section that the numerical
method is uniquely ergodic. Rather we have shown that any stationary measure
of the numerical system is close to the true stationary measure in the $\rho
-$distance. It is, of course, possible in many settings to show that the
numerical method is itself uniquely ergodic (cf.
\cite{Talay90,TalayHam02,MattinglyStuartHighamSDENUM02}).

\begin{remark}
It follows from Theorem~\ref{thm:expan} that under its assumptions the error
$\rho(\mu^{\Delta},\mu)$ can be expanded in powers of the time step:
\[
\rho(\mu^{\Delta},\mu)=C_{0}^{p}\Delta^{p}+\cdots+C_{q}\Delta^{p+q}+O\left(
\Delta^{p+q+1}\right)  .
\]

\end{remark}

\begin{remark}
It is also worth contrasting the result of Theorem~\ref{cor:invMeasureClose}
with a short alternative proof. Let $P_{t}$ denote the Markov semigroup
associated with the SDE and $P_{n}^{\Delta}$ with $n$ steps of a numerical
method with step size $\Delta$. Suppose one knows for some metric $d$ on
probability measures that $d(P_{t}\mu_{1},P_{t}\mu_{2})\leq Ke^{-\gamma
t}d(\mu_{1},\mu_{2})$ holds for all probability measures $\mu_{i}$ and that
for any fixed $n$, $d(P_{n}^{\Delta}\mu^{\Delta},P_{n\Delta}\mu^{\Delta})\leq
C_{n}\Delta^{p}$ for some constant $C_{n},$ all $\Delta$ sufficiently small
and any measure $\mu^{\Delta}$ invariant for $P^{\Delta}$. Then if one fixes
an $n$ so that $Ke^{-\gamma n}\leq1/2$ then
\begin{align*}
d(\mu,\mu^{\Delta})  &  =d(P_{n\Delta}\mu,P_{n}^{\Delta}\mu^{\Delta})\leq
d(P_{n\Delta}\mu,P_{n\Delta}\mu^{\Delta})+d(P_{n\Delta}\mu^{\Delta}%
,P_{n}^{\Delta}\mu^{\Delta})\\
&  \leq\frac{1}{2}d(\mu,\mu^{\Delta})+C_{n}\Delta^{p}%
\end{align*}
and collecting the $d(\mu,\mu^{\Delta})$ produces the estimate $d(\mu
,\mu^{\Delta})\leq2C_{n}\Delta^{p}$ for all $\Delta$ sufficiently small. The
challenge in implementing such a seemingly simple program is obtaining the two
estimates in the same metric $d$. Typically, the first estimate is available
in the total variation distance. On the other hand, the second estimate is
usually estimated in distances requiring test functions with a number of
derivatives. The recent works \cite{HaiMat08:??,hairerMattingly2009} allow one
to obtain the first estimate in the 1-Wasserstein distance which simplifies
matching the two norms. Using this strategy on can obtain a bound in a
stronger norm (such as total variation or 1-Wasserstein metric), but the
$d(P_{n}^{\Delta}\mu,P_{\Delta n}\mu)$ convergence rate will often be
sub-optimal in these metrics. On the other hand, one can obtain $d(P_{t}%
\mu_{1},P_{t}\mu_{2})\leq Ke^{-\gamma t}d(\mu_{1},\mu_{2})$ with $d$ defined
as in the metric $\rho$ above with $H=W^{2(p+1),\infty}$ by using that fact
that $\Vert P_{1}\phi\Vert_{H}\leq C\Vert\phi\Vert_{\infty}$ from some $C$ and
hence $\Vert P_{t+1}\phi-\bar{\phi}\Vert_{H}\leq C\Vert P_{t}\phi-\bar{\phi
}\Vert_{\infty}\leq CKe^{-\gamma t}\Vert\phi-\bar{\phi}\Vert_{\infty}\leq
CKe^{-\gamma t}\Vert\phi-\bar{\phi}\Vert_{H}$. This gives a result comparable
to the hypoelliptic result in Theorem~\ref{cor:invMeasureClose} though misses
the smoothing in the elliptic result which is embodied in the fact one can use
$H=W^{2p,\infty}$. (Some partial smoothing could have been extracted in the
hypoelliptic case in Theorem~\ref{cor:invMeasureClose} with more care). See
\cite{HaiMat08:??} for this program executed in a particular setting.
\end{remark}

\subsection{Relationship to Stein's method}

\label{sec:stein} This section is devoted to outlining the similarity between
the current setting and Stein's method.\footnote{When this work was nearing
completion, a conversation with between JCM and Sourav Chatterjee prompted the
authors to write this section reflecting on the relation between their
approach and Stein's method.} This connection is not fully explored in this
work but we believe that it is insightful to highlight the main idea. Though
our goals are different, there are some
%AMS reword above
passing similarities between the details in this paper and
\cite{Chatterjee07,Barbour90}.

Let us recall that Stein's method is a generic tool for finding bounds on a
distance between two distributions which then can give quantitative
convergence results. Denoting the target distribution as $\pi$, the idea is to
find an operator $\mathcal{A}$ and a determining class of functions
$\mathcal{G}$ so that if, for all $g\in\mathcal{G}$ one has $\int%
\mathcal{A}g(x)d\widetilde{\pi}(x)=0$, then this implies that $\widetilde{\pi
}=\pi$.\footnote{As a referee correctly observed, the
  Echeverria-Weiss theorem is useful in identifying  when such a
  condition characterizes an invariant measure and identifying the
  limiting invariant measure for a sequence measures with  $\int
\mathcal{A}g(x)d\widetilde{\pi_n}(x) \rightarrow 0$ as $n \rightarrow
\infty$. However, here we are really interested in the next order
question. How to use the degree to which the characterizing equation
is \textit{not} satisfied to obtain quantitative  estimates of the
convergence rate.}Given any $h$ sufficiently nice, one next solves Stein's equation
\begin{equation}
\mathcal{A}g=h\label{stein}%
\end{equation}
with the tacit assumption that the solution $g$ exists and lies in
$\mathcal{G}$. Observe that a basic solvability condition on the above
equation requires that $\int hd\pi=0$ when $\mathcal{A}^{\ast}\pi=0$. In this
setting we can consider the equation $\mathcal{A}g=h-\int hd\pi$ for more
general, uncentered $h$. In order to quantify the distance of a given measure
$\widetilde{\pi}$ from $\pi$, one tries to control $\int(\mathcal{A}%
g)(z)d\widetilde{\pi}(z)$ by some norm of $g$ which can in turn be controlled
by an appropriate norm of $h$. For definiteness let us assume that
\[
\int(\mathcal{A}g)(z)d\widetilde{\pi}(z)\leq\epsilon|g|_{\mathcal{G}}%
\leq\epsilon C\Vert h\Vert
\]
if $g$ satisfies \eqref{stein}. Then
\[
\int hd\widetilde{\pi}-\int hd\pi=\int(\mathcal{A}g)(z)d\widetilde{\pi}%
(z)\leq\epsilon C\Vert h\Vert\;.
\]
By taking the supremum over all $h$ in a given class with $\Vert h\Vert\leq1,$
one obtains control over the distance of $\widetilde{\pi}$ from $\pi$. In
particular, if $\epsilon$ is small then $\widetilde{\pi}$ is close to $\pi$.

This is essentially the methodology we have followed. Taking $\mathcal{A}$ to
be the generator $\mathcal{L}$ of the Markov process, we are ensured that
$\mathcal{A}^{\ast}\pi=0$ if $\pi$ is the Markov process' stationary
measure.\footnote{This is done in some versions of Stein's method. See
\cite{Barbour90}.} The basic idea of this note is to show that if
$\mathcal{\widetilde{L}}$ is a generator of another Markov process so that
$\mathcal{\widetilde{L}}-\mathcal{L}$ is small then any stationary measure of
the second Markov process will be close to $\pi$. We will further assume that
$\widetilde{\pi}$ is ergodic.

While our paper concerns a mixture of continuous and discrete time, the idea
can be more easily demonstrated in a continuous time setting. Let
$\mathcal{P}_{t}$ and $\mathcal{\widetilde{P}}_{t}$ be strong Markov
semigroups on $\mathbf{R}^{d}$ with generators $\mathcal{L}$ and
$\mathcal{\widetilde{L}}$ both defined on some common domain $D$. Let $\pi$
and $\widetilde{\pi}$ be stationary measures for $\mathcal{P}_{t}$ and
${\tilde{\mathcal{P}}}_{t},$ respectively. Assume that for some set of bounded
functions $\mathcal{G} \subset D$ there exists an $\epsilon>0$ so that
\begin{equation}
\int(\mathcal{L}-\mathcal{\widetilde{L}})gd\widetilde{\pi}\leq\epsilon\Vert
g\Vert_{\mathcal{G}} \label{eq:closeLs}%
\end{equation}
for any $g\in\mathcal{G}$. Further assume that for some class of bounded, real
valued functions $\mathcal{H}$ on $\mathbf{R}^{d}$ one can solve
$\mathcal{L}g=h-\int hd\pi$ for $h\in\mathcal{H}$ with a solution
$g\in\mathcal{G}$ and $\Vert g\Vert_{\mathcal{G}}\leq K\Vert h\Vert
_{\mathcal{H}}$ for a fixed constant $K$. Then for all $h\in\mathcal{H}$,
\[
\int hd\widetilde{\pi}-\int hd\pi=\int(\mathcal{L}g)d\widetilde{\pi}%
=\int(\mathcal{\widetilde{L}}g)d\widetilde{\pi}+\int(\mathcal{L}%
-\mathcal{\widetilde{L}})gd\widetilde{\pi}\leq\epsilon K\Vert h\Vert
_{\mathcal{H}}\ ,
\]
where in moving from the penultimate expression to the last, we have used
$\mathcal{\widetilde{L}}^{\ast}\widetilde{\pi}=0$. Since the right-hand side
is uniform in $h$ we can take the supremum over $h$. If $\mathcal{H}$ was a
rich-enough class to define a metric, we obtain some estimate of the distance
between the two stationary measures in that metric.

If one does not have good control over $\widetilde{\pi}$ then
\eqref{eq:closeLs} can be difficult to obtain. Instead it is often easier to
replace \eqref{eq:closeLs} with
\begin{equation}
\label{eq:steinBound}\limsup_{T\rightarrow\infty} \frac1T \int_{0}^{T}
\widetilde{P}_{t}(\mathcal{ \widetilde{L}} - \mathcal{ L})g(x) dt \leq
\epsilon\|g\|_{\mathcal{G}}%
\end{equation}
for any $g\in\mathcal{G}$ and $\widetilde{\pi}$-a.e. $x$. As before, we assume
that for some class of functions $\mathcal{H}$ from $\mathbf{R}^{d}%
\rightarrow\mathbf{R}$, one can solve $\mathcal{L}g=h-\int hd\pi$ for
$h\in\mathcal{H}$ with a solution $g\in\mathcal{G}$ such that $\Vert
g\Vert_{\mathcal{G}}\leq K\Vert h\Vert_{\mathcal{H}}$. Observe that if
$\mathcal{G}$ is a class of bounded functions then the second assumption is
always satisfied.

Continuing since $\widetilde{P}_{t}\mathcal{L}g(x)=\widetilde{P}_{t}h(x)-\int
hd\pi$,
\begin{align*}
\widetilde{P}_{T}g(x)-g(x)  &  =\int_{0}^{T}\widetilde{P}_{t}%
\mathcal{\widetilde{L}}g(x)dt\\
&  =\int_{0}^{T}\widetilde{P}_{t}(\mathcal{\widetilde{L}-L})g(x)dt+\int%
_{0}^{T}\widetilde{P}_{t}h(x)dt-T\int_{\mathbf{R}^{d}}hd\pi\;.
\end{align*}
Rearranging, dividing by $T$, and using $|\widetilde{P}_{T}g(x)|\leq
|g|_{\infty}$, produces
\[
\int_{\mathbf{R}^{d}}hd\pi-\frac{1}{T}\int_{0}^{T}\widetilde{P}_{t}%
h(x)dt\leq\frac{2\Vert g\Vert_{\infty}}{T}+\frac{1}{T}\int_{0}^{T}%
\widetilde{P}_{t}(\mathcal{\widetilde{L}-L})g(x)dt\ .
\]
By Birkoff's ergodic theorem, we know that for $\widetilde{\pi}$-almost every
$x$
\[
\lim_{T\rightarrow\infty}\frac{1}{T}\int_{0}^{T}\widetilde{P}_{t}%
h(x)dt=\int_{\mathbf{R}^{d}}hd\widetilde{\pi}\ ,
\]
so from \eqref{eq:steinBound} and $\Vert g\Vert_{\mathcal{G}}\leq K\Vert
h\Vert_{\mathcal{H}}$ we obtain
\[
\int_{\mathbf{R}^{d}}hd\pi-\int_{\mathbf{R}^{d}}hd\widetilde{\pi}\leq\epsilon
K\Vert h\Vert_{\mathcal{H}}%
\]
for any $h\in\mathcal{H}$. Hence in the language of
Section~\ref{sec:ClosenessII}, one has $\rho(\pi,\tilde{\pi})\leq\epsilon K$.

This argument can be modified in many ways, the restriction that $\mathcal{G}$
and $\mathcal{H}$ are classes of bounded functions can be removed by assuming
some control over $\widetilde{P}_{t}g(x)$ uniform in time. If that control can
be maintained using a function which is integrable with respect to
$\widetilde{\pi}$ then the requirement that $\widetilde{\pi}$ be ergodic can
be removed. Although we do not fully explore these issues here, we believe
that the connections made in this subsection are useful.
%AMS reword previous sentence

\section{Variance of the Empirical Time Average}

%AMS reword next sentence
\label{sec:practice} Theorems~\ref{thm:veryBasicConv} and \ref{thm:moreConv}
are important for implementing time-averaging in computational practice. There
are three types of errors arising in computing stationary averages: (i)
numerical integration error (estimated by $K\Delta$ in (\ref{eq:bias2}) and
(\ref{eq:mm}) and by $K\Delta^{p}$ in (\ref{eq:bias3}) and (\ref{eq:mm2}));
(ii) the error due to the distance from the stationary distribution (i.e., the
error due to the finite time of integration $T$ estimated by $K/T$ in
(\ref{eq:bias1}), (\ref{eq:bias2}) and (\ref{eq:bias3})); (iii) the
statistical error. The first two errors contribute to the bias of the
estimator $\hat{\phi}_{N}$ (see (\ref{eq:bias2}) and (\ref{eq:bias3})). The
error of numerical integration is controlled by the time step and the choice
of a method. It can be estimated in practice using the Talay-Tubaro expansion
(see Section~\ref{sec:expan} and \cite{TalayTubaro90,MilsteinTretyakov2004})
in the usual fashion. The statistical error is contained in the second term of
(\ref{eq:mm}) and related to the variance of the estimator $\hat{\phi}_{N}$
(see the details below and also (\ref{eq:var1})). Both the error due to the
finite time of integration $T$ and the statistical error are controlled by the
choice of the integration time $T$ (or what is the same, the choice of the
number of steps $N$ under fixed time step $\Delta$). They correspond to
properties of the continuous dynamics of the SDE (\ref{eq:sde}) and they are
almost independent (assuming a sufficiently small $\Delta)$ of a method or
time step $\Delta$ used.

Let us consider the statistical error. Theorems~\ref{thm:veryBasicConv}
and~\ref{thm:basicConv} immediately imply that $\mathrm{Var}(\hat{\phi}%
_{N})=O(\Delta^{2}+1/T).$ In order to get a more accurate estimate for the
variance of $\hat{\phi}_{N},$ analogous to the one in the continuous case
(\ref{eq:var1}), i.e., to have the statistical error of the estimator
$\hat{\phi}_{N}$ controlled by $T$ only, we need to make use a mixing-type condition.

The Markov process $X(t)$ defined on the torus $\mathbf{T}^{d}$ by the SDE
(\ref{eq:sde}) is uniformly ergodic, it has exponential mixing rates (see,
e.g., \cite{DMT95,MeynTweedie,Bhattacharya1985}), and there are some positive
$C$ and $\gamma$ such that for any $t>0$ and $\theta>0$ the inequality
\begin{equation}
\left\vert \mathbf{E}\phi(X(t))\phi(X(t+\theta))-\mathbf{E}\phi
(X(t))\mathbf{E}\phi(X(t+\theta))\right\vert \leq Ce^{-\gamma\theta}
\label{eq:mix}%
\end{equation}
holds. Further, as it has been already mentioned before, it is possible in
many settings to show that the Markov chain $X_{n}$ generated by a numerical
method is also geometrically ergodic (cf.
\cite{Talay90,TalayHam02,MattinglyStuartHighamSDENUM02,KlokovVeretennikov2006}%
). Then, due to the compactness of the phase space, the chain is uniformly
ergodic and has an exponential mixing like in (\ref{eq:mix}) \cite[Chapter
16]{MeynTweedie}. We note that in the cited papers one of the conditions to
ensure the ergodicity of $X_{n}$ is the requirement that the numerical method
uses random variables with densities positive everywhere. We do not address
here the question about mixing rate for the Markov chains $X_{n}$ generated by
more general numerical methods treated in this paper leaving it for further
study but instead we assume that the following relaxed mixing condition is
satisfied for $X_{n}$ and a $\phi\in H$ and $l>0:$
\begin{equation}
\left\vert \mathbf{E}\phi(X_{k})\phi(X_{k+l})-\mathbf{E}\phi(X_{k}%
)\mathbf{E}\phi(X_{k+l})\right\vert \leq\frac{K}{\left(  t_{k+l}-t_{k}\right)
^{2}}\ , \label{eq:deco}%
\end{equation}
where $K>0$ is a constant independent of $\Delta,$ $k,$ $l.$ This condition is
much weaker than (\ref{eq:mix}) but it is sufficient for the proof of the
following proposition. At the same time, a faster decorrelation than
(\ref{eq:deco}) does not improve the estimate (\ref{eq:var2}).

\begin{proposition}
\label{prp:var}In the setting of Theorem~\ref{thm:veryBasicConv} and under the
condition (\ref{eq:deco}), one has
\begin{equation}
\mathrm{Var}(\hat{\phi}_{N})\leq\frac{K}{T}, \label{eq:var2}%
\end{equation}
where $T=N\Delta$ and $K>0$ is a constant independent of $\Delta$ and $T$.
\end{proposition}

\begin{proof}
It follows from \eqref{sumedUp} that
\begin{gather}
\mathrm{Var}(\hat{\phi}_{N})\leq K\frac{\mathbf{E}\left\vert \psi
_{N}-\mathbf{E}\psi_{N}\right\vert ^{2}}{T^{2}}\label{eq:pp1}\\
+\frac{K}{N^{2}}\sum_{i=0}^{N-1}\sum_{j=0}^{N-1}\mathbf{E}\Big[  \left(
(\mathcal{L}^{\Delta}-\mathcal{L})\psi_{i}-\mathbf{E}(\mathcal{L}^{\Delta
}-\mathcal{L})\psi_{i}\right)  \left(  (\mathcal{L}^{\Delta}-\mathcal{L}%
)\psi_{j}-\mathbf{E}(\mathcal{L}^{\Delta}-\mathcal{L})\psi_{j}\right)  \Big]
\nonumber\\
+\frac{1}{T^{2}}\sum_{i=0}^{2}\mathbf{E}\left\vert S_{i,N}-\mathbf{E}%
S_{i,N}\right\vert ^{2}+\frac{1}{T^{2}}\sum_{i=0}^{3}\mathbf{E}|M_{i,N}%
|^{2}.\nonumber
\end{gather}
We have used here that $M_{i,N}$ are martingales. Using the estimates for
$\mathbf{E}|M_{i,N}|^{2}$ from Theorem~\ref{thm:veryBasicConv}, we get that
$\sum_{i=0}^{3}\mathbf{E}|M_{i,N}|^{2}\leq KT.$

To estimate the terms with $S_{i,N}$ and $(\mathcal{L}^{\Delta}-\mathcal{L}%
)\psi$ in (\ref{eq:pp1}), we exploit the condition (\ref{eq:deco}). For
example, consider the term with $S_{1,N}.$ We have
\begin{align*}
\mathbf{E}\left\vert S_{1,N}-\mathbf{E}S_{1,N}\right\vert ^{2}  &
=\frac{\Delta^{4}}{4}\mathbf{E}\Big( \sum_{n=0}^{N-1}\left(  D^{2}\psi
_{n}[F_{n},F_{n}]-\mathbf{E}D^{2}\psi_{n}[F_{n},F_{n}]\right)  \Big) ^{2}\\
&  =\frac{\Delta^{4}}{4}\sum_{n=0}^{N-1}\mathbf{E}\Big( \left(  D^{2}\psi
_{n}[F_{n},F_{n}]-\mathbf{E}D^{2}\psi_{n}[F_{n},F_{n}]\right)  \Big) ^{2}\\
&  +\frac{\Delta^{4}}{2}\sum_{i=0}^{N-1}\sum_{j=i+1}^{N-1}\mathbf{E}%
\Big( \left(  D^{2}\psi_{i}[F_{i},F_{i}]-\mathbf{E}D^{2}\psi_{i}[F_{i}%
,F_{i}]\right) \\
&  \times\left(  D^{2}\psi_{j}[F_{j},F_{j}]-\mathbf{E}D^{2}\psi_{j}%
[F_{j},F_{j}]\right)  \Big)\\
&  \leq K\Delta^{3}T+K\Delta^{4}\sum_{i=0}^{N-1}\sum_{j=i+1}^{N-1}\frac
{1}{\left(  t_{j}-t_{i}\right)  ^{2}}\leq K\Delta T.
\end{align*}

Analyzing analogously the other terms in (\ref{eq:pp1}), we arrive at the
stated result. \medskip
\end{proof}

Although we proved Proposition~\ref{prp:var} for the method (\ref{eq:euler}),
it is also valid for a more general class of numerical methods from
Section~\ref{sec:high}. \medskip

In practice one usually estimates the statistical error of $\hat{\phi}_{N}$ as
follows\footnote{Of course, better statistical estimators can be used to
quantify the statistical error of the time averaging but we restrict ourselves
here to a simple one.}. We run a long trajectory $MT$ split into $M$ blocks of
a large length $T=\Delta N$ each. We evaluate the estimators $_{m}\hat{\phi
}_{N},$ $m=1,\ldots,M,$ for each block. Since $T$ is big and a time decay of
correlations is usually fast, $_{m}{\hat{\phi}}_{N}$ can be considered as
almost uncorrelated. We compute the sampled variance
\[
\hat{D}=\frac{1}{M-1}\sum_{m=1}^{M}\big({_{m}{\hat{\phi} }_{N} }\big)^{2} -
\Big( \frac{1}{M}\sum_{m=1}^{M} {_{m}}{\hat{\phi}}_{N}\Big)^{2}.
\]
For a sufficiently large $T$ and $M$, $\mathbf{E}\,\hat{\phi}_{N}$ belongs to
the confidence interval
\[
\mathbf{E}\,\hat{\phi}_{N}\in\big(\hat{\phi}_{NM}-c {\textstyle\frac
{\sqrt{\hat{D}}}{\sqrt{M}}},\hat{\phi}_{NM}+c{\textstyle\frac{\sqrt{\hat{D}}%
}{\sqrt{M}}}\big)\ ,
\]
with probability, for example $0.95$ for $c=2$ and $0.997$ for $c=3.$ Note
that $\mathbf{E}\,\hat{\phi}_{N}$ contains the two errors forming the bias as
explained at the beginning of this section. We also pay attention to the fact
that $\hat{D}\sim1/T$ (cf. (\ref{eq:var2})), i.e., it is inverse proportional
to the product $\Delta N.$

\begin{remark}
Instead of time averaging considered in this paper, stationary averages can be
computed using ensemble averaging, i.e., by the following estimate for the
stationary average $\bar{\phi}$:
\begin{equation}
\bar{\phi}\approx\mathbf{E}\phi(X(t))\approx\mathbf{E}\phi(\bar{X}%
(t))\approx\check{\phi}\overset{\mbox{\tiny def}}{=}\frac{1}{L}\sum_{l=1}%
^{L}\phi\big(\bar{X}^{(l)}(t)\big)\ ,\nonumber
\end{equation}
where $t$ is a sufficiently large time, $\bar{X}$ is an approximation of $X,$
and $L$ is the number of independent approximate realizations. The total error
$R_{\check{\phi}}\overset{\mbox{\tiny def}}{=}\check{\phi}-\bar{\phi}$
consists of three parts: the error $\varepsilon$ of the approximation
$\bar{\phi}$ by $\mathbf{E}\phi(X(t))$; the error of numerical integration
$C\Delta^{p}$ (here $p$ is the weak order of the method); and the Monte Carlo
error which is proportional to $1/\sqrt{L}.$ More specifically
\[
\left\vert \mathrm{Bias}(\check{\phi})\right\vert =\left\vert \mathbf{E}%
\check{\phi}-\bar{\phi}\right\vert \leq K\Delta^{p}+\varepsilon
\ ,\ \ \ \ \ \mathrm{Var}(\check{\phi})=O(1/L)\ .
\]
Here, in comparison with the time-averaging approach, each error is controlled
by its own parameter: sufficiently large $t$ ensures smallness of
$\varepsilon;$ time step $\Delta$ (as well as the choice of a numerical
method) controls the numerical integration error; the statistical error is
regulated by choosing an appropriate number of independent trajectories $L.$
For further details see \cite{MilsteinTretyakov2007}.
\end{remark}

\section*{\textbf{Acknowledgments}}

JCM thanks the NSF for its support, through DMS-0616710 and DMS-0449910, as
well as the support of a Sloan Fellowship. AMS is grateful to the ERC and
EPSRC for financial support. MVT was partially supported by\ the EPSRC
Research Grant EP/D049792/1. JCM thanks Martin Harier and Sourav Chatterjee
for useful discussions and the hospitality of the Warwick Mathematics
Institute where the core of this work was done.
The authors thank an anonymous referee and 
Denis Talay for useful suggestions which
improved the manuscript.

%\section{Convergence in Total-Variation Norm}

%We now strengthen the results in Corollary \ref{cor:invMeasureClose} by
%measuring the distance in the total-variation norm. To do this we will
%begin by strengthening Theorem~\ref{thm:veryBasicConv} under the
%following additional assumptions.

\bibliographystyle{siam}
\bibliography{./refs}

\end{document}